\documentclass[11pt]{article}
\usepackage{listings}
\usepackage{pdflscape}
\usepackage{amsfonts}
\usepackage{epsfig}
\usepackage{lscape}
\usepackage{longtable}
\usepackage{amsmath,amssymb,amsthm}
\usepackage{graphicx}
\usepackage{subfigure}
\usepackage{color}
\usepackage{placeins}
\usepackage{url}
\usepackage{cases}
\usepackage{longtable}
\usepackage{supertabular}
\usepackage{multirow}
\newcommand{\tabincell}[2]{\begin{tabular}{@{}#1@{}}#2\end{tabular}}
\usepackage{amsmath}
\allowdisplaybreaks[4]
\usepackage{cite}
\usepackage{algorithmic}
\usepackage{algorithm,float}
\usepackage{booktabs}



\newtheorem{theorem}{Theorem}
\newtheorem{assumption}{Assumption}

\newtheorem{remark}{Remark}
\newtheorem{definition}{Definition}
\newtheorem{proposition}{Proposition}

\begin{document}

\title{\bf A dual Newton based preconditioned proximal point algorithm for exclusive lasso models\footnotemark[1]}
\author{Meixia Lin\footnotemark[2], \quad Defeng Sun\footnotemark[3], \quad Kim-Chuan Toh\footnotemark[4], \quad Yancheng Yuan\footnotemark[5]}
\date{December 06, 2019}
\maketitle

\renewcommand{\thefootnote}{\fnsymbol{footnote}}
\footnotetext[1]{{\bf Funding:} Defeng Sun is supported in part by Hong Kong Research Grant Council grant PolyU153014/18p and Kim-Chuan Toh by ARF grant R146-000-257-112 of the Ministry of Education of Singapore.}
\footnotetext[2]{Department of Mathematics, National University of Singapore, 10 Lower Kent Ridge Road, Singapore ({\tt lin\_meixia@u.nus.edu}).}
\footnotetext[3]{Department of Applied Mathematics, The Hong Kong Polytechnic University, Hung Hom, Hong Kong ({\tt defeng.sun@polyu.edu.hk}).}
\footnotetext[4]{Department of Mathematics and Institute of Operations Research and Analytics, National University of Singapore, 10 Lower Kent Ridge Road, Singapore ({\tt mattohkc@nus.edu.sg}).}
\footnotetext[5]{Department of Mathematics, National University of Singapore, 10 Lower Kent Ridge Road, Singapore ({\tt yuanyancheng@u.nus.edu}).}
\renewcommand{\thefootnote}{\arabic{footnote}}

\begin{abstract}
	The exclusive lasso (also known as elitist lasso) regularization has become popular recently due to its superior performance on group sparsity. Compared to the group lasso regularization which enforces the competition on variables among different groups, the exclusive lasso regularization also enforces the competition within each group. In this paper, we propose a highly efficient dual Newton based preconditioned proximal point algorithm (PPDNA) to solve machine learning models involving the exclusive lasso regularizer. As an important ingredient, we provide a rigorous proof for deriving the closed-form solution to the proximal mapping of the weighted exclusive lasso regularizer. In addition, we derive the corresponding HS-Jacobian to the proximal mapping and analyze its structure --- which plays an essential role in the efficient computation of the PPA subproblem via applying a semismooth Newton method on its dual. Various numerical experiments in this paper demonstrate the superior performance of the proposed PPDNA against other state-of-the-art numerical algorithms.
\end{abstract}

\\[5pt]

\section{Introduction}
\label{sec:intro}
Structured sparsity is very important in feature learning, not only for avoiding over-fitting, but also in making the model more interpretable. Many regularizers and their combinations have been proposed to enforce sparsity for parameterized machine learning models. The most popular ones among them are probably the classical lasso \cite{tibshirani1996regression} and the group lasso \cite{yuan2006model} regularizers. Lasso, group lasso and their variants have been extensively studied in terms of both their statistical properties \cite{tibshirani1996regression,yuan2006model,zou2006adaptive} and efficient numerical computations
\cite{beck2009fast,efron2004least,li2018highly,zhang2018efficient}. The classical lasso model has been important in enforcing sparsity on variables while performing feature selection. However, there is no structure enforced in the sparsity pattern. Instead, the group lasso is known to enforce the sparsity at an inter-group level, where variables from different groups compete to be selected.

However, in some real applications, in addition to the unstructured sparsity (e.g. lasso) or the inter-group level structured sparsity (e.g. group lasso), we also need the intra-group level sparsity. That is, not only features from different groups, but also features in a seemingly cohesive group are competing to survive. One real example comes from building an index exchange-traded fund (index ETF) to track a specific index in the stock market. To diversify the risk across different sectors, we need to perform portfolio selection both across and within sectors, which indeed means that we also need the intra-group level sparsity. To achieve this, a new regularizer called the exclusive lasso has been proposed in  \cite{kong2014exclusive,zhou2010exclusive} (also named as elitist lasso \cite{kowalski2009sparse}). Let $w\in \mathbb{R}_{++}^n$ be a weight vector and $\mathcal{G} := \{g_1,\cdots,g_l\}$ be a partition of variable index groups such that $\bigcup_{i=1}^l g_i = \{1, 2, \dots, n\}$ and $g_i \bigcap g_j = \emptyset$ for any $i\neq j $. Then the corresponding  weighted exclusive lasso regularizer is defined as
\begin{equation}
\Delta^{\mathcal{G},w}(x) :=\sum_{i=1}^l \|w_{g_i} \circ x_{g_i}\|_1^2, \quad \forall\, x \in \mathbb{R}^n,\label{eq: exclusive-lasso-regularizer}
\end{equation}
where ``$\circ$'' denotes the Hadamard product, and $x_{g_i}$ denotes the sub-vector extracted from $x$ based on the index set $g_i$. As indicated in the above expression, a squared $\ell_2$-norm  is applied to   different groups, and a weighted $\ell_1$-norm is used to enforce sparsity within each group. Naturally, when solving exclusive lasso models, we can expect that each $x_{g_i}$ is nonzero under some conditions, which means that every group has representatives.

The exclusive lasso regularizer was first proposed for multi-task learning in \cite{zhou2010exclusive}, and has been widely used in applications such as image processing \cite{zhang2016robust}, sparse feature clustering \cite{yamada2017localized} and NMR spectroscopy \cite{campbell2017within}. Some algorithms have been proposed for solving models involving the exclusive lasso regularizer, such as the smooth method based on accelerated proximal gradient (APG) \cite{zhang2016robust}, the iterative least squares algorithm (ILSA) \cite{kong2014exclusive,yamada2017localized}, and the coordinate descent method \cite{campbell2017within}. However, some popular algorithmic frameworks like APG \cite{nesterov2013gradient}, FISTA \cite{beck2009fast} and alternating direction method of multipliers (ADMM) \cite{eckstein1992douglas,glowinski1975approximation} have not been used to solve these kind of problems. The main reason may lie in the fact   that the closed-form solution to the proximal mapping ${\rm Prox}_{\rho\|w \circ \cdot\|_1^2}(\cdot)$ is unknown to all the previous works, although it has already been proposed in \cite{kowalski2009sparse}. In order to adopt a proximal gradient method to solve the exclusive lasso model, Campbell et al. \cite{campbell2017within} used an iterative subroutine to compute ${\rm Prox}_{\rho\|\cdot\|_1^2}(\cdot)$ with uniform weights.

In this paper, we recap the closed-form solution derived for the proximal mapping of the exclusive lasso regularizer in \cite{kowalski2009sparse}. As the derivation in \cite{kowalski2009sparse} is not completely rigorous\footnote{The author uses the gradient of the  exclusive regularizer to derive his formula while the latter is not differentiable in general. See a counterexample in Section \ref{subsec:prox}.}, here we provide a rigorous proof based on a quadratic programming reformulation of the proximal mapping minimization problem and the corresponding Karush-Kuhn-Tucker (KKT) conditions. As mentioned above, such a closed-form solution can be directly used in some popular algorithmic frameworks such as APG and ADMM for solving the exclusive lasso models. However, based on our numerical experiments, it is very challenging for these algorithms to solve large scale exclusive lasso problems.

To overcome the numerical challenges in large scale cases, we design a highly efficient second-order type algorithm, the dual Newton based preconditioned proximal point algorithm (PPDNA), to solve the exclusive lasso model. As a key ingredient for PPDNA, we derive the HS-Jacobian of ${\rm Prox}_{\rho\|w \circ \cdot\|_1^2}(\cdot)$. We also conduct numerical experiments to demonstrate the superior performance of PPDNA for solving popular machine learning models with the exclusive lasso regularizer, comparing to other state-of-the-art algorithms mentioned above.

We summarize our main contributions in this paper as follows.
\begin{enumerate}
	\item We develop a dual Newton based preconditioned proximal point algorithm (PPDNA) to solve  machine learning models involving the exclusive lasso regularizer.
	\item We provide a rigorous proof for the closed-form solution to the proximal mapping of $\rho\|w\circ\cdot\|_1^2$ and derive the corresponding generalized Jacobian. These results are critical for the computational efficiency of various algorithmic frameworks.
	\item We demonstrate numerically that PPDNA is highly efficient and robust when comparing to ILSA, APG and ADMM, even with the closed-form proximal mapping of the exclusive lasso regularizer. Furthermore, we apply the exclusive lasso model in index ETF and achieve better out-of-sample results, comparing to the lasso and group lasso models.
\end{enumerate}

The remaining parts of the paper are organized as follows.  In Section \ref{sec:pppa}, we propose the preconditioned proximal point algorithm (preconditioned PPA) for solving general 2-block convex composite programming problems. The dual Newton algorithm (DNA) for solving the PPA subproblem is introduced in Section \ref{sec:dna}. In Section \ref{sec:proxJacobian}, we provide a rigorous proof for the closed-form solution to ${\rm Prox}_{\rho\|w \circ \cdot\|_1^2}(\cdot)$, followed by the derivation of the corresponding HS-Jacobian, which is an essential ingredient for designing a semismooth Newton method to solve the dual of the PPA subproblems. In Section \ref{sec:numerical}, we present our numerical results when solving regression problems and classification problems, on both synthetic data and real applications. In the end, we conclude the paper.

\vspace{0.3cm}
\noindent\textbf{Notations and preliminaries:} For any $z\in \mathbb{R}$, ${\rm sign}(z)$ is defined to be $1$ if $z\geq 0$, and $-1$ otherwise. For $x\in \mathbb{R}$, denote $x^{+}=\max\{x,0\}$ and $x^{-}=\min\{x,0\}$. We use ``${\rm Diag}(x)$" to denote the diagonal matrix whose diagonal is given by the vector $x$. For any self-adjoint positive semidefinite linear operator ${\cal M}:\mathbb{R}^n\rightarrow \mathbb{R}^n$, we define $\langle x,x'\rangle_{\cal M}:=\langle x,{\cal M}x'\rangle$, and $\|x\|_{\cal M}:=\sqrt{\langle x,x\rangle_{\cal M}}$ for all $x,x'\in \mathbb{R}^n$. For a given subset ${\cal C}$ of $\mathbb{R}^n$, we denote the weighted distance of $x\in \mathbb{R}^n$ to ${\cal C}$ as ${\rm dist}_{\cal M}(x,{\cal C}):=\inf_{x'\in {\cal C}}\|x-x'\|_{\cal M}$. The largest (smallest) eigenvalue of ${\cal M}$ is denoted as $\lambda_{\max}({\cal M})$ ($\lambda_{\min}({\cal M})$).

For any closed proper convex function $p:\mathbb{R}^n\rightarrow (-\infty,\infty]$, the conjugate function is defined as $p^*(z):=\sup_{x\in\mathbb{R}^n}\{\langle x,z\rangle-p(x)\}$. The Moreau envelope of $p$ at $x$ is defined by
\[
{\rm E}_p(x):=\min_{y\in \mathbb{R}^n}\Big\{ p(y)+\frac{1}{2}\|y-x\|^2\Big\},
\]
and the associated proximal mapping ${\rm Prox}_p(x)$ is defined as the unique solution of the above problem. It is known that $\nabla {\rm E}_p(x)=x-{\rm Prox}_p(x)$ and ${\rm Prox}_p(x)$ is Lipschitz continuous with modulus $1$ \cite{moreau1965proximite,rockafellar1976monotone}.

\section{A preconditioned proximal point algorithm for solving the exclusive lasso problem}
\label{sec:pppa}
The exclusive lasso problem is a special case of the general 2-block convex composite programming problem, which is given as
\begin{equation}
\label{eq: Convex-composite-program}
\min_{x\in \mathbb{R}^n} \Big\{ f(x):= h({\cal A} x)  - \langle c,x\rangle + p(x)\Big\},
\end{equation}
where ${\cal A}: \mathbb{R}^n \to \mathbb{R}^m$ is a linear mapping, $c\in \mathbb{R}^n$, $h: \mathbb{R}^m \to \mathbb{R}$ is a convex loss function, and $p:\mathbb{R}^n\rightarrow (-\infty,+\infty]$ is a closed, proper, convex function. In particular, if $p(\cdot)=\lambda \Delta^{\mathcal{G},w}(\cdot)$, where $\Delta^{\mathcal{G},w}(\cdot)$ is the exclusive lasso regularizer defined in \eqref{eq: exclusive-lasso-regularizer} and $\lambda>0$ is a parameter, \eqref{eq: Convex-composite-program} reduces to the so-called exclusive lasso model.

Define the proximal residual function $R:\mathbb{R}^n\rightarrow \mathbb{R}^n$ by
\begin{align}
R(x)=x-{\rm Prox}_p(x-{\cal A}^* \nabla h({\cal A}x)+c),\quad \forall x\in \mathbb{R}^n,\label{eq:residual}
\end{align}
and the set-valued map ${\cal T}_f(x):=\partial f(x)$. Assume that the solution set $\Omega$ to \eqref{eq: Convex-composite-program} is nonempty. The first order optimality condition of \eqref{eq: Convex-composite-program} implies that $\bar{x}\in \Omega$ if and only if $R(\bar{x})=0$.

The proximal point algorithm (PPA) \cite{rockafellar1976augmented,rockafellar1976monotone} is a well established algorithmic framework for solving convex minimization problems, which is proven to have an asymptotic superlinear convergence rate. Recently, Li et al. extend the classical PPA to the preconditioned proximal point algorithm (preconditioned PPA) in \cite{li2019asymptotically}. In this section, we apply the preconditioned PPA to solve the general 2-block convex composite programming problem.

\subsection{Preconditioned PPA for 2-block convex composite programming problems}
For any starting point $x^0\in \mathbb{R}^n$, the preconditioned PPA generates a sequence $\{x^k\}\subseteq \mathbb{R}^n$ by the following approximate rule for solving \eqref{eq: Convex-composite-program}:
\begin{align}
x^{k+1}\approx {\cal P}_k(x^k)=\arg\min_{x\in\mathbb{R}^n}\Big\{h({\cal A} x)  - \langle c,x\rangle + p(x)+\frac{1}{2\sigma_k}\|x-x^k\|_{{\cal M}_k}^2\Big\},\label{eq:pre_ppa}
\end{align}
where $\{\sigma_k\}$ is a sequence of nondecreasing positive real numbers $(\sigma_k\uparrow \sigma_{\infty}\leq \infty)$ and $\{{\cal M}_k\}$ is a given sequence of self-adjoint positive definite linear operators satisfying:
\[
{\cal M}_k\succeq {\cal M}_{k+1},\quad {\cal M}_k\succeq \lambda_{\min}I_n,\quad \forall k\geq 0,
\]
with some constant $\lambda_{\min}>0$. Observe that when ${\cal M}_k\equiv I_n$ for all $k\geq 0$, the preconditioned PPA reduces to the classical PPA.

To ensure the convergence of the preconditioned PPA, we need the following stopping criteria as proposed in \cite{li2019asymptotically}:
\begin{align}
\|x^{k+1}-{\cal P}_k(x^k)\|_{{\cal M}_k}&\leq \epsilon_k,\quad \epsilon_k \geq 0,\quad \sum_{k=0}^{\infty}\epsilon_k <\infty,
\tag{A}\label{eq:stopA_pre_ppa}\\
\|x^{k+1}-{\cal P}_k(x^k)\|_{{\cal M}_k}&\leq \delta_k\|x^{k+1}-x^k\|_{{\cal M}_k},\quad 0\leq \delta_k < 1,\quad \sum_{k=0}^{\infty}\delta_k <\infty.\tag{B}\label{eq:stopB_pre_ppa}
\end{align}

\subsection{Convergence of the preconditioned PPA}
We adopt the convergence results of the preconditioned PPA, which can be found in \cite{li2019asymptotically}, for the convenience of the readers.
\begin{theorem}\label{thm:convergence_ALM}
	(1) Let $\{x^k\}$ be the sequence generated by the preconditioned PPA \eqref{eq:pre_ppa} with the stopping criterion \eqref{eq:stopA_pre_ppa}. Then $\{x^k\}$ is bounded and
	\[
	{\rm dist}_{{\cal M}_{k+1}}(x^{k+1},\Omega)\leq {\rm dist}_{{\cal M}_{k}}(x^{k},\Omega)+\epsilon_k,\quad \forall k\geq 0,
	\]		
	where $\Omega$ is the solution set of \eqref{eq: Convex-composite-program}. In addition, $\{x^k\}$ converges to some $x^*\in \Omega$.\\
	\noindent(2) Let $r:=\sum_{i=0}^{\infty}\epsilon_k+{\rm dist}_{{\cal M}_0}(x^0,\Omega)$. Assume that for this $r>0$, there exists a constant $\kappa>0$ such that ${\cal T}_f(x)$ satisfies the following error bound assumption
	\begin{align}
	{\rm dist}(x,\Omega)\leq \kappa {\rm dist}(0,{\cal T}_f(x)),\quad \mbox{$\forall x\in \mathbb{R}^n$ satisfying }{\rm dist}(x,\Omega)\leq r. \label{error_bound}
	\end{align}
	Suppose that $\{x^k\}$ is generated by the preconditioned PPA with the stopping criteria \eqref{eq:stopA_pre_ppa} and \eqref{eq:stopB_pre_ppa}. Then it holds for all $k\geq 0$ that
	\begin{equation}\label{eq:rate_pppa}
	{\rm dist}_{{\cal M}_{k+1}}(x^{k+1},\Omega)\leq \mu_k {\rm dist}_{{\cal M}_{k}}(x^{k},\Omega),
	\end{equation}
	where
	\[
	\mu_k=\frac{1}{1-\delta_k}\frac{\delta_k+(1+\delta_k)\kappa\lambda_{\max}({\cal M}_k
		)}{\sqrt{\sigma_k^2+\kappa^2\lambda_{\max}^2(
			{\cal M}_k)}}\rightarrow \mu_{\infty}=\frac{\kappa\lambda_{\infty}}{\sqrt{
			\sigma_{\infty}^2+\kappa^2\lambda_{\infty}^2}}<1,\quad k\rightarrow \infty,
	\]
	where $\lambda_{\infty}=\lim_{k\rightarrow \infty}\lambda_{\max}({\cal M}_k)$. In addition, it holds that for all $k\geq 0$,
	\[
	{\rm dist}(x^{k+1},\Omega)\leq \frac{\mu_k}{\sqrt{\lambda_{\min}({\cal M}_{k+1})}}{\rm dist}_{{\cal M}_{k}}(x^{k},\Omega).
	\]
\end{theorem}

The convergence rate of the preconditioned PPA relies on the error bound conditions of ${\cal T}_f$. The following proposition can be used to establish error bound conditions for many commonly used loss function plus piecewise linear-quadratic regularizer, which is an application of \cite[Theorem 2]{zhou2017unified}.
\begin{proposition}\label{prop:luotseng}
	Assume that $\Omega$ is non-empty and compact.
	Suppose that (1) $h$ is continuously differentiable on $\mathbb{R}^m$ and strongly convex on any compact convex set in $\mathbb{R}^m$; (2) $p(\cdot)$ is a piecewise linear-quadratic convex function. Then  for any $\xi\geq \inf_{x\in \mathbb{R}^n}f(x)$, there exist constants $\kappa,\varepsilon>0$ such that
	\begin{align*}
	{\rm dist}(x,\Omega)\leq \kappa \|R(x)\|\mbox{ \quad  for all $x\in \mathbb{R}^n$ with $f(x)\leq \xi$, $\|R(x)\|\leq \varepsilon$},
	\end{align*}
	where $R(x)$ is defined in  \eqref{eq:residual}.
\end{proposition}
\begin{proof}
	From \cite[Proposition 1]{zhou2017unified}, we know that that there exists a $\bar{y}\in \mathbb{R}^m$ such that
	\[
	{\cal A}x=\bar{y},\quad {\cal A}^* \nabla h({\cal A}x)-c =\bar{g},\quad \forall x\in \Omega,
	\]
	where $\bar{g}={\cal A}^* \nabla h(\bar{y})-c$. Consider the collection ${\cal C}:=\{\Gamma_h(\bar{y}),\Gamma_p(\bar{g})\}$, where
	\begin{align*}
	\Gamma_h(y):=\{x\in \mathbb{R}^n \mid {\cal A}x=y\},\quad \Gamma_p(g):=\{x\in \mathbb{R}^n \mid -g\in\partial p(x)\}.
	\end{align*}
	Since $\Gamma_h(\bar{y})$ is the set of solutions to a linear system, it is a polyhedral closed convex set. According to \cite[Corollary 23.5.1]{rockafellar1970convex},
	\begin{align*}
	\Gamma_p(\bar{g}):=\{x\in \mathbb{R}^n\mid -\bar{g}\in\partial p(x)\}=\{x\in \mathbb{R}^n\mid x\in\partial p^*(-\bar{g})\}=\partial p^*(-\bar{g}).
	\end{align*}
	Since $p$ is piecewise linear-quadratic, $p^*$ is also piecewise linear-quadratic by \cite[Theorem 11.14(b)]{rockafellar2009variational}. Thus $\partial p$ and $\partial p^*$ are both polyhedral due to \cite[Proposition 10.21]{rockafellar2009variational}. Therefore, $\Gamma_h(\bar{y})$ and $\Gamma_p(\bar{g})$ are closed convex polyhedral sets. By \cite[Corollary 3]{bauschke1999strong}, we know that ${\cal C}$ is boundedly linearly regular. Since $\partial p^*$ is a polyhedral multi-function, we can see from \cite[Proposition 3H.1]{dontchev2009implicit} that $\partial p^*$ is calm at $-\bar{g}$ for any $\bar{x}\in \Omega$, thus $\partial p=(\partial p^*)^{-1}$ is metrically subregular at $\bar{x}$ for $-\bar{g}$ by \cite[Theorem 3H.3]{dontchev2009implicit}. Therefore, by \cite[Theorem 2]{zhou2017unified}, the solution map $\Gamma(y,g):=\{x\in \mathbb{R}^n\mid {\cal A}x=y,-g\in \partial p(x)\}$ is calm at $(\bar{y},\bar{g})$ for any $\bar{x}\in\Omega$. Then the desired conclusion holds by \cite[Corollary 1]{zhou2017unified}.
\end{proof}

For the exclusive lasso regularized models, we can see from the next proposition that the error bound assumption \eqref{error_bound} holds for the linear regression problem and the logistic regression problem, which means that the preconditioned PPA can be expected to have fast linear convergence when the parameters $\{\sigma_k\}$ are large.
\begin{proposition}
	Assume that in \eqref{eq: Convex-composite-program}, $p(\cdot)=\lambda \Delta^{\mathcal{G},w}(\cdot)$. Then the error bound assumption \eqref{error_bound} holds in the following two cases:
	\begin{itemize}
		\item[(1)] $h(y)=\sum_{i=1}^m (y_i-b_i)^2/2$, for some given vector $b\in \mathbb{R}^m$;
		\item[(2)] $h(y) = \sum_{i=1}^m \log(1 + \exp(-b_i y_i))$, for some given vector $b\in \{-1,1\}^m$.
	\end{itemize}
\end{proposition}
\begin{proof}
	(1) When $h(y)=\sum_{i=1}^m (y_i-b_i)^2/2$, $f(\cdot)$ is a piecewise linear-quadratic convex function, from \cite{sun1986monotropic}, ${\cal T}_f(\cdot)$ is piecewise polyhedral, thus it satisfies the error bound assumption \eqref{error_bound} \cite{li2019asymptotically,robinson1981some}.
	
	(2) When $h(y) = \sum_{i=1}^m \log(1 + \exp(-b_i y_i))$, since $f(\cdot)$ is nonnegative, and $f(x)\rightarrow +\infty$ as $\|x\|\rightarrow +\infty$, $\Omega$ is non-empty and compact. Given $r>0$, define $\Omega_r:=\{x\in \mathbb{R}^n\mid {\rm dist}(x,\Omega)\leq r\}$. Due to the fact that $\Omega$ is compact, $\Omega_r$ is compact and thus $\xi:=\max_{x\in \Omega_r}f(x)$ is finite. From Proposition \ref{prop:luotseng}, we know that for this $\xi$, there exist constants $\kappa,\varepsilon>0$ such that
	\begin{align}
	{\rm dist}(x,\Omega)\leq \kappa \|R(x)\|\mbox{ \quad  for all $x\in\mathbb{R}^n$ with $f(x)\leq \xi$, $\|R(x)\|\leq \varepsilon$},\label{eq:luocondition}
	\end{align}
	where $R(x)$ is defined as in \eqref{eq:residual}. We consider two cases:\\
	Case 1: $x\in \Omega_r$ and $\|R(x)\|\leq \varepsilon$. From \eqref{eq:luocondition}, we have ${\rm dist}(x,\Omega)\leq \kappa \|R(x)\|$.\\
	Case 2: $x\in \Omega_r$ and $\|R(x)\|> \varepsilon$. Then ${\rm dist}(x,\Omega)\leq (r/\varepsilon)\varepsilon \leq (r/\varepsilon) \|R(x)\|$.
	
	Therefore, it holds that
	\[
	{\rm dist}(x,\Omega)\leq \max\{\kappa,(r/\varepsilon)\} \|R(x)\|,\quad \forall x\in \Omega_r.
	\]
	Next, we follow the ideas in \cite[Theorem 3.1]{dong2009extension} and \cite[Proposition 2.4]{cui2016asymptotic}. Let $y\in {\cal T}_f(x)$, which means that $y\in  {\cal A}^* \nabla h({\cal A}x)-c+\partial p(x)$, or equivalently
	$x={\rm Prox}_p(x+y- {\cal A}^* \nabla h({\cal A}x)+c)$. Thus
	\begin{align*}
	\|R(x)\|=\|{\rm Prox}_p(x-{\cal A}^* \nabla h({\cal A}x)+c)-{\rm Prox}_p(x+y-{\cal A}^* \nabla h({\cal A}x)+c)\|\leq \|y\|.
	\end{align*}
	As a result,
	\[
	{\rm dist}(x,\Omega)\leq \max\{\kappa,(r/\varepsilon)\} \|y\|,\quad \forall y\in {\cal T}_f(x)\mbox{ and }
	x\in \Omega_r.
	\]
	Therefore,
	$
	{\rm dist}(x,\Omega)\leq \max\{\kappa,(r/\varepsilon)\} {\rm dist}(0,{\cal T}_f(x)),\; \forall x\in \Omega_r.$
\end{proof}

Note that the key challenge in executing the preconditioned PPA is whether the nonsmooth problem \eqref{eq:pre_ppa} can be solved efficiently. We consider two special cases. The first case is ${\cal M}_k\equiv I_n$ for all $k\geq 0$, and the other is ${\cal M}_k\equiv I_n+\tau {\cal A}^*{\cal A}$, where $\tau>0$ is a given positive number. To efficiently solve the preconditioned PPA subproblems, we design a dual Newton algorithm (DNA) to solve \eqref{eq:pre_ppa}, where the algorithm is superlinearly (or even quadratically) convergent when the functions $h(\cdot)$ and $p(\cdot)$ in \eqref{eq:pre_ppa} satisfy suitable conditions that are stated later in Theorem \ref{thm:convergence_SSN}.

\section{A dual Newton algorithm for solving the preconditioned PPA subproblem}
\label{sec:dna}
For all $k\geq 0$, we aim to solve the preconditioned PPA subproblem
\begin{equation}\label{eq:sub_pppa}
\min_{x\in \mathbb{R}^n}\Big\{ f_k(x):=h({\cal A} x)  - \langle c,x\rangle + p(x)+\frac{1}{2\sigma_k}\|x-x^k\|_{{\cal M}_k}^2\Big\}.
\end{equation}
Obviously, $f_k(\cdot)$ is a strongly convex, nonsmooth function, which is not necessarily Lipschitz. Thus the above minimization problem admits a unique solution $\bar{x}^{k+1}$. The main point is how one can solve \eqref{eq:sub_pppa} in a fast and robust way. Our choice is the dual Newton algorithm (DNA) as already explained at the end of last section.

\subsection{The case when ${\cal M}_k\equiv I_n$}
In this classical case, one can write \eqref{eq:sub_pppa} equivalently as
\begin{equation}\label{eq:case1_P}
\min_{x\in \mathbb{R}^n,y\in \mathbb{R}^m}\Big\{ h(y)  - \langle c,x\rangle + p(x)+\frac{1}{2\sigma_k}\|x-x^k\|^2\mid {\cal A}x-y=0\Big\}.
\end{equation}
The dual of the above problem, after ignoring the constant term, is
\begin{align}\label{eq:case1_D}
\max_{u\in \mathbb{R}^n} \Big\{\phi_k(u):=-h^*(u) - \frac{1}{2\sigma_k}\|x^k+\sigma_k c-\sigma_k {\cal A}^*u\|^2 +\frac{1}{\sigma_k}{\rm E}_{\sigma_k p}(x^k+\sigma_k c-\sigma_k{\cal A}^*u)\Big\}.
\end{align}
Suppose that the following assumption holds for $h^*$.
\begin{assumption}\label{ass_fstar}
	$h^*(\cdot)$ is twice continuously differentiable and strongly convex with modulus $\alpha_h$ in ${\rm int}({\rm dom}(h^*))$.
\end{assumption}
Note that when we consider the least squares loss function $h(y)=\sum_{i=1}^m (y_i-b_i)^2/2$, Assumption \ref{ass_fstar} holds with $\alpha_h = 1$. Under Assumption \ref{ass_fstar}, we can see that $\phi_k(\cdot)$ is strongly concave, thus \eqref{eq:case1_D} has a unique optimal solution $\bar{u}^{k+1}$, and $\bar{x}^{k+1}$ can be obtained by
\[
\bar{x}^{k+1}={\rm Prox}_{\sigma_k p}(x^k+\sigma_k c-\sigma_k {\cal A}^*\bar{u}^{k+1}).
\]

Now we can give the full description of the preconditioned PPA with subproblems solved by the DNA in Algorithm \ref{alg:ppdna}.
\begin{algorithm}[!h]\small
	\caption{Dual Newton based preconditioned PPA (PPDNA) for \eqref{eq: Convex-composite-program}}
	\label{alg:ppdna}
	\begin{algorithmic}
		\STATE {\bfseries Initialization:} Choose $x^0\in \mathbb{R}^n$, $\sigma_0 > 0$. For $k = 0, 1, 2, \dots$
		\REPEAT
		\STATE {\bfseries Step 1}. Compute
		\begin{align}
		u^{k+1} &\approx \arg\max \phi_{k}(u),\label{eq:solve_u}\\
		x^{k+1} &= {\rm Prox}_{\sigma_k p}(x^k+\sigma_k c-\sigma_k {\cal A}^*u^{k+1}),
		\end{align}
		where $\phi_k(\cdot)$ is defined as in \eqref{eq:case1_D}, to satisfy the stopping criteria \eqref{eq:stopA_pre_ppa} and \eqref{eq:stopB_pre_ppa}.
		\\[3pt]
		\STATE {\bfseries Step 2}. Update $\sigma_{k+1} \uparrow \sigma_{\infty} \leq \infty$.
		\UNTIL{Stopping criterion is satisfied.}
	\end{algorithmic}
\end{algorithm}

As one can see in the algorithm, we need the implementations of the stopping criteria \eqref{eq:stopA_pre_ppa} and \eqref{eq:stopB_pre_ppa} associated with $u^{k+1}$ and $x^{k+1}$. By the discussions in \cite{li2018highly,luque1984asymptotic,rockafellar1976augmented}, the stopping criteria \eqref{eq:stopA_pre_ppa} and \eqref{eq:stopB_pre_ppa} can be achieved by the following implementable criteria when Assumption \ref{ass_fstar} holds:
\begin{align*}
\|\nabla \phi_k(u^{k+1})\|&\leq \sqrt{\alpha_h/\sigma_k}\epsilon_k,\quad \epsilon_k \geq 0,\quad \sum_{k=0}^{\infty}\epsilon_k <\infty,\tag{A'}\label{eq:stopA'_pre_ppa}\\
\|\nabla \phi_k(u^{k+1})\|&\leq \sqrt{\alpha_h/\sigma_k} \delta_k \|x^{k+1}-x^k\|,\quad 0\leq \delta_k < 1,\quad \sum_{k=0}^{\infty}\delta_k <\infty,\tag{B'}\label{eq:stopB'_pre_ppa}
\end{align*}
where $\nabla \phi_k(u) =  -\nabla h^{*}(u) + {\cal A}  {\rm Prox}_{\sigma_k p}(x^k + \sigma_k c- \sigma_k {\cal A}^{*}u)$.

Next we discuss how to solve \eqref{eq:solve_u} in Algorithm \ref{alg:ppdna}. For fixed $\sigma>0$, $\tilde{x}\in \mathbb{R}^n$, we aim to solve
\[
\min_{u\in \mathbb{R}^n}\Big\{ \phi(u):=-h^*(u) - \frac{1}{2\sigma}\|\tilde{x}+\sigma c-\sigma {\cal A}^*u\|^2 +\frac{1}{\sigma}{\rm E}_{\sigma p}(\tilde{x}+\sigma c-\sigma {\cal A}^*u)\Big\}.
\]
Since $\phi(\cdot)$ is continuously differentiable, it is equivalent to solving the nonsmooth equation
\begin{equation}\label{eq: newton-system}
\nabla\phi(u)=-\nabla h^{*}(u) + {\cal A}  {\rm Prox}_{\sigma p}(\tilde{x} + \sigma c- \sigma{\cal A}^{*}u).
\end{equation}
Note that $\nabla \phi(\cdot)$ is Lipschitz continuous, but not differentiable. Due to the quadratic convergence  of Newton's method, it is usually the first choice for solving a nonlinear equation. However, the direct application of Newton's method to \eqref{eq: newton-system} is infeasible since the function $\nabla \phi(\cdot)$ is nonsmooth. Fortunately, the semismooth version of the Newton's method has been established in \cite{kummer1988newton,qi1993nonsmooth}. This allows us to solve \eqref{eq: newton-system} by a semismooth Newton method (SSN), which has at least superlinear convergence. The concept of semismoothness can be found in the supplementary materials.

We now derive the generalized Jacobian of the Lipschitz continuous function $\nabla\phi(\cdot)$. For given $u$, the following set-valued map is well defined:
\[
\hat{\partial}^2 \phi(u)  :=  -\nabla^2 h^{*}(u) - \sigma{\cal A}\partial{\rm Prox}_{\sigma p}(\tilde{x} + \sigma c - \sigma{\cal A}^{*}u){\cal A}^{*},
\]
where $\partial {\rm Prox}_{\sigma p}(\tilde{x} + \sigma c- \sigma{\cal A}^{*}u)$ is the  generalized Jacobian of the Lipschitz continuous mapping  ${\rm Prox}_{\sigma p}(\cdot)$  at $\tilde{x} + \sigma c- \sigma{\cal A}^{*}u$. Then we can treat $\hat{\partial}^2 \phi(u)$ as the surrogate generalized Jacobian of $\nabla\phi(\cdot)$ at $u$.

Now we present our semismooth Newton (SSN) method in Algorithm \ref{alg:ssn} for solving \eqref{eq: newton-system}, which can be expected to get at least a superlinear (or even quadratic) convergence rate.
\begin{algorithm}[!h]\small
	\caption{Semismooth Newton method (SSN) for \eqref{eq: newton-system}}
	\label{alg:ssn}
	\begin{algorithmic}
		\STATE {\bfseries Initialization:} Given $u^{0} \in  {\rm int}({\rm dom}(h^*))$, $\mu \in (0, 1/2)$, $\tau \in (0, 1]$, and $\bar{\gamma}, \delta \in (0, 1)$. For $j = 0, 1, \dots$
		\REPEAT
		\STATE {\bfseries Step 1}. Select an element ${\cal H}_j \in \hat{\partial}^{2} \phi(u^j)$. Apply the direct method or the conjugate gradient (CG) method to find an approximate solution $d^j \in \mathbb{R}^m$ to
		\begin{equation}\label{eq: cg-system}
		{\cal H}_j(d^j) \approx - \nabla \phi(u^j)
		\end{equation}
		such that $\|{\cal H}_j(d^j) + \nabla\phi(u^j)\| \leq \min(\bar{\gamma}, \|\nabla\phi(u^j)\|^{1+\tau})$.
		\\[3pt]
		\STATE {\bfseries Step 2}. Set $\alpha_j = \delta^{m_j}$, where $m_j$ is the smallest nonnegative integer $m$ for which
		$$\phi(u^j + \delta^m d^j) \leq \phi(u^j) + \mu\delta^m \langle \nabla\phi(u^j), d^j \rangle .$$
		\\[3pt]
		\STATE{\bfseries Step 3}. Set $u^{j+1} = u^j + \alpha_j d^j$.
		\UNTIL{Stopping criterion based on $u^{j+1}$ is satisfied.}
	\end{algorithmic}
\end{algorithm}
Theorem \ref{thm:convergence_SSN} gives the convergence result of the SSN method.
\begin{theorem} \label{thm:convergence_SSN}
	Supppose that Assumption \ref{ass_fstar} holds. Assume that for any $\sigma >0$, ${\rm Prox}_{\sigma p} (\cdot)$ is strongly semismooth with respect to $\partial{\rm Prox}_{\sigma p} (\cdot)$. Let $\{u^j\}$ be the sequence generated by Algorithm \ref{alg:ssn}. Then $\{u^j\}$ converges to the unique solution $\bar{u}$ of the problem \eqref{eq: newton-system}, and for $j$ sufficiently large,
	\[
	\|u^{j+1} - \bar{u}\| = O(\|u^{j} - \bar{u}\|^{1+\tau}),
	\]
	where $\tau \in (0, 1]$ is given in the algorithm.
\end{theorem}
\begin{proof}
	Due to the strong convexity of $h^*$, all the elements in $\hat{\partial}^2\phi(u)$ for all $u$ are negative definite, and $V\preceq -\alpha_h I$ for any $V\in\hat{\partial}^2\phi(u)$ with $\alpha_h$ given in Assumption \ref{ass_fstar}. By \cite[Proposition 3.3 and Theorem 3.4]{zhao2010newton}, we can see that $\{u^j\}$ converges to the unique solution $\bar{u}$. Then by mimicking the proof of \cite[Theorem 3]{li2018efficiently}, we can get the convergence rate of   $\{u^j\}$.
\end{proof}

\begin{remark}
	For the case of $p(\cdot)=\lambda \Delta^{\mathcal{G},w}(\cdot)$, it will be proved in the next section that ${\rm Prox}_{\sigma p}(\cdot)$ is strongly semismooth with respect to $\partial_{\rm HS} {\rm Prox}_{\sigma p}(\cdot)$, where $\partial_{\rm HS} {\rm Prox}_{\sigma p}(\cdot)$ is the HS-Jacobian of ${\rm Prox}_{\sigma p}(\cdot)$.
\end{remark}

We should emphasize that the efficiency in computing the Newton direction in \eqref{eq: cg-system} depends critically on exploiting the sparsity structure of the generalized Jacobian of ${\rm Prox}_{\sigma p}(\cdot)$. The case of $p(\cdot)=\lambda \Delta^{\mathcal{G},w}(\cdot)$ will be discussed in the next section, where an important property called the second-order sparsity is carefully treated in the implementation.

\subsection{The case when ${\cal M}_k\equiv I_n+\tau {\cal A}^*{\cal A}$ }
In some problems, Assumption \ref{ass_fstar} on $h^*$ may not hold, e.g. when $h(y)=\|y-b\|_2$ for a given vector $b\in \mathbb{R}^m$. Then we can choose ${\cal M}_k\equiv I_n+\tau {\cal A}^*{\cal A}$, where $\tau$ is a given positive number. The reason why we add the ${\cal A}^*{\cal A}$ term is to deal with the possible lack of strong convexity in the function $h^*$.

To be specified, \eqref{eq:sub_pppa} can be equivalently written as
\begin{equation}\label{eq:ATA_P}
\min_{x\in\mathbb{R}^n,y\in\mathbb{R}^m}\Big\{ h(y)  - \langle c,x\rangle +p(x)+\frac{1}{2\sigma_k}\|x-x^k\|^2+\frac{\tau}{2\sigma_k}\|y-{\cal A}x^k\|^2\mid {\cal A}x-y=0\Big\}.
\end{equation}
As discussed before, we can solve \eqref{eq:ATA_P} by the dual Newton algorithm. The dual of \eqref{eq:ATA_P} is given as
\begin{align}
\max_{u\in \mathbb{R}^n}\Big\{\psi_k(u)&:= -\frac{\tau}{2\sigma_k}\|{\cal A}x^k+\frac{\sigma_k}{\tau}u\|^2 +\frac{\tau}{\sigma_k}{\rm E}_{\sigma_k h/\tau}({\cal A}x^k+\frac{\sigma_k}{\tau}u) +\frac{\tau}{2\sigma_k}\|{\cal A}x^k\|^2\notag\\
&- \frac{1}{2\sigma_k}\|x^k+\sigma_k c-\sigma_k {\cal A}^*u\|^2 +\frac{1}{\sigma_k}{\rm E}_{\sigma_k p}(x^k+\sigma_k c-\sigma_k {\cal A}^*u)+\frac{1}{2\sigma_k}\|x^k\|^2\Big\}.\label{eq:ATA_D}
\end{align}
As long as we can obtain $\bar{u}^{k+1}\in \arg\max\psi_k(u)$, the update of $x$ in the preconditioned PPA will be obtained by
\[
\bar{x}^{k+1}={\rm Prox}_{\sigma_k p}(x^k + \sigma_k c - \sigma_k {\cal A}^{*}\bar{u}^{k+1}).
\]
Therefore, one can still apply the general algorithmic framework PPDNA in Algorithm \ref{alg:ppdna} to solve \eqref{eq: Convex-composite-program} but with $\phi_k(\cdot)$ in \eqref{eq:solve_u} replaced by $\psi_k(\cdot)$ in \eqref{eq:ATA_D}. The following proposition shows that the stopping criteria \eqref{eq:stopA_pre_ppa} and \eqref{eq:stopB_pre_ppa} can be achieved by using $u^{k+1}$ and $x^{k+1}$. The idea is come from \cite{liu2012implementable} and the proof can be found in the supplementary materials.

\begin{proposition}\label{prop:stop_of_ATA}
	When we use preconditioned PPA to solve \eqref{eq: Convex-composite-program} with ${\cal M}_k\equiv I_n+\tau {\cal A}^*{\cal A}$, where $\tau$ is a given positive number, the stopping criteria \eqref{eq:stopA_pre_ppa} and \eqref{eq:stopB_pre_ppa} can be achieved by the following two implementable ones:
	\begin{align}
	f_k(x^{k+1})-\psi_k(u^{k+1})&\leq \frac{\epsilon_k^2}{2\sigma_k},\quad \epsilon_k \geq 0,\quad \sum_{k=0}^{\infty}\epsilon_k <\infty,\tag{A''}\label{eq:stopA_ATA}\\
	f_k(x^{k+1})-\psi_k(u^{k+1})&\leq \frac{\delta_k^2}{2\sigma_k} \|x^{k+1}-x^k\|_{{\cal M}_k}^2,\quad 0\leq \delta_k < 1,\quad \sum_{k=0}^{\infty}\delta_k <\infty,\tag{B''}\label{eq:stopB_ATA}
	\end{align}
	where $f_k(\cdot)$ is defined in \eqref{eq:sub_pppa} and $\psi_k(\cdot)$ is defined in \eqref{eq:ATA_D}.
\end{proposition}

Next we discuss about how to solve \eqref{eq:ATA_D}. As one can see, $\psi_k$ is continuously differentiable with
\[
\nabla \psi_k(u)=-{\rm Prox}_{\sigma_k h/\tau}({\cal A}x^k+\frac{\sigma_k}{\tau}u)+{\cal A} {\rm Prox}_{\sigma_k p}(x^k+\sigma_k c-\sigma_k{\cal A}^*u).
\]
The surrogate generalized Jacobian of the Lipschitz continuous function $\nabla\psi_k(\cdot)$ at $u$ can be defined as
\[
\hat{\partial}^2 \psi_k(u)  :=  -\frac{\sigma_k}{\tau}\partial {\rm Prox}_{\sigma_k h/\tau}({\cal A}x^k+\frac{\sigma_k}{\tau}u)  - \sigma_k{\cal A}\partial{\rm Prox}_{\sigma_k p}(x^k + \sigma_k c - \sigma_k{\cal A}^{*}u){\cal A}^{*}.
\]
Under some conditions on $h$, e.g. the elements in $\partial {\rm Prox}_{\nu h}(\cdot)$ are positive definite for any $\nu >0$, one can still apply the SSN method in Algorithm \ref{alg:ssn} to solve \eqref{eq:ATA_D} just as in the previous subsection.

\begin{remark}
	Suppose we consider the logistic regression problem, i.e. $h(y) = \sum_{i=1}^m \log(1 + \exp(-b_i y_i))$, for some given vector $b\in \{-1,1\}^m$. Since it can be proved that $h^*(\cdot)$ satisfies Assumption \ref{ass_fstar}, it is natural for us to apply the classical PPA (preconditioned PPA with ${\cal M}_k\equiv I_n$). Besides, we can also apply the preconditioned PPA with ${\cal M}_k\equiv I_n+\tau {\cal A}^*{\cal A}$. The main motivation for considering the latter case is that the condition number of the linear system in the SSN method would not blow up while those associated with the former case may blow up when $|({\cal A}x)_i|$ is large for some $i$. As for the proximal mapping of $h$, it can be computed coordinate-wise by Newton's method efficiently.
\end{remark}

\section{Closed-form solution to the proximal mapping of $\rho\|w\circ\cdot\|_1^2$ and its generalized Jacobian}
\label{sec:proxJacobian}
From the discussion in the previous section in solving the exclusive lasso model with the PPDNA algorithm, it is clear that we need the proximal mapping ${\rm Prox}_p(\cdot)$ and its generalized Jacobian for $p(\cdot)=\lambda \Delta^{\mathcal{G},w}(\cdot)$. In this section, for a given weight vector $w\in \mathbb{R}^n_{++}$ and $\rho>0$, we derive the closed-form solution to ${\rm Prox}_{\rho\|w\circ\cdot\|_1^2}(\cdot)$ and its generalized Jacobian.

\subsection{Closed-form solution to ${\rm Prox}_{\rho\|w\circ\cdot\|_1^2}(\cdot)$}
\label{subsec:prox}
The closed-form solution for the proximal mapping of $\rho\|w\circ\cdot\|_1^2$ we present here is consistent with the result in \cite[Proposition 4]{kowalski2009sparse}. However, in section 4.1 of \cite{kowalski2009sparse}, after a change of variables, the author tries to find the optimal solution of a constrained optimization problem by directly setting the gradient to zero (equations (23) and (24) in \cite{kowalski2009sparse}). Although the formula obtained for the closed-form solution is fortuitously correct, the derivation is not mathematically rigorous as the exclusive lasso regularizer is not continuously differentiable. One can use a simple example to demonstrate the gap. Consider the problem
\begin{align*}
\min_{x_{1,1},x_{1,2}}\Big\{\frac{1}{2}(x_{1,1}-1)^2+\frac{1}{2}(x_{1,2}-0.5)^2+(|x_{1,1}|+|x_{1,2}|)^2\Big\}.
\end{align*}
The true solution is $x^*=[1/3;0]$. But equation (23) in \cite{kowalski2009sparse} is equivalent to 
\begin{align*}
|x_{1,1}| = 1-2(|x_{1,1}|+|x_{1,2}|),\quad
|x_{1,2}| = 0.5-2(|x_{1,1}|+|x_{1,2}|).
\end{align*}
Thus $|x_{1,1}|=2/5$, $|x_{1,2}|=-1/10$. But the latter contradicts the fact that $|x_{1,2}|\geq 0$. 

Our derivations in Proposition \ref{prop:pos_prox} thus aim to provide a rigorous proof based on the KKT optimality conditions. We first consider the case when $a\geq 0$, then one can show that ${\rm Prox}_{\rho\|w\circ\cdot\|_1^2}(a)$ must also be nonnegative and hence it could be equivalently computed by
\begin{align}
x(a)&:=\arg\min_{x \in \mathbb{R}^n_{+}} \Big\{\frac{1}{2}\|x - a\|^2 + \rho\|w\circ x\|_1^2
\Big\}=\arg\min_{x \in \mathbb{R}^n_{+}} \Big\{\frac{1}{2}\|x - a\|^2 +  \rho x^T(ww^T)x \Big\}.\label{eq:pos_problem}
\end{align}
Note that since the objective function is strongly convex, the above minimization problem has a unique solution, which can be computed as in the following proposition.
\begin{proposition}\label{prop:pos_prox}
	Given $\rho > 0$ and $a \in \mathbb{R}_{+}^n \backslash\{0\}$. Let $a^{w}\in \mathbb{R}^n$ be defined as $a^{w}_i :=a_i/w_i$, for $i=1,\cdots,n$. There exists a permutation matrix $\Pi$ such that $\Pi a^{w}$ is sorted in a non-increasing order. Denote $\tilde{a}=\Pi a $,  $\tilde{w}=\Pi w $, and
	\[
	s_{i} = \sum_{j=1}^{i}
	\tilde{w}_j \tilde{a}_j, \quad
	L_{i} = \sum_{j=1}^{i}\tilde{w}_j^2,\quad
	\alpha_i = \frac{  s_{i}}{1 + 2\rho L_{i}}, \quad i = 1, 2, \dots, n.
	\]
	Let $\bar{\alpha}=\max_{1\leq i \leq n}\alpha_i$. Then, $x(a)$ defined in \eqref{eq:pos_problem} can be computed as: $x(a) = (a -2\rho \bar{\alpha} w)^+$.
\end{proposition}
\begin{proof}
	The KKT conditions for \eqref{eq:pos_problem} are given by
	\begin{equation}\label{eq: pos_KKT}
	x - a + 2\rho w w^Tx + \mu = 0,\;
	\mu \circ x =0,\; \mu \leq 0, \;x \geq 0,
	\end{equation}
	where $\mu \in \mathbb{R}^n$ is the dual multiplier. If $(x^*,\mu^{*})$ satisfies the KKT conditions \eqref{eq: pos_KKT},  by denoting $\beta=w^T x^*$, we can see that
	\begin{equation*}
	x^* +  \mu^* = a - 2\rho\beta w ,\;
	\mu^* \circ x^* =0,\; \mu^* \leq 0, \;x^* \geq 0.
	\end{equation*}
	Therefore, $(x^*,\mu^{*})$ have the representations:
	\[
	x^*=(a - 2\rho\beta w)^+,\quad \mu^*=(a - 2\rho\beta w)^-.
	\]
	Then our aim is to find the value of $\beta$. By the definition of $\beta$, we can see that
	\begin{align*}
	\beta = \sum_{i=1}^n w_i x^*_i=\sum_{i=1}^n w_i (a_i - 2\rho\beta w_i)^+=\sum_{i=1}^n w^2_i ((a^{w})_i - 2\rho\beta )^+=\sum_{i=1}^n \tilde{w}^2_i ((\Pi a^{w})_i - 2\rho\beta )^+.
	\end{align*}
	Note that there must exist $j$ such that $(\Pi  a^{w})_j>2\rho \beta$, otherwise, we have $\beta=0$ and $\Pi a^{w}\leq 0$ (equivalent to $a\leq 0$), which contradicts the assumption. Since $\Pi a^{w}$ is sorted in a non-increasing order, there exists an index $k$ such that $\tilde{a}_{1}/\tilde{w}_{1} \geq \dots \geq \tilde{a}_{k}/\tilde{w}_{k} \geq 2\rho\beta > \tilde{a}_{k+1}/\tilde{w}_{k+1} \geq \dots \geq \tilde{a}_{n}/\tilde{w}_{n}$. Therefore,
	\[
	\beta = \sum_{i=1}^{k} \tilde{w}^2_i ((\Pi a^{w})_i - 2\rho\beta )=\sum_{i=1}^{k}  \tilde{w}_i\tilde{a}_i- 2\rho\beta \sum_{i=1}^{k}  \tilde{w}^2_i=s_{k} -2\rho \beta L_{k} ,
	\]
	which means that
	\[
	\beta = \frac{s_{k} }{1+2\rho L_{k} }=\alpha_k.
	\]
	Next we show that $\beta=\bar{\alpha}$, which means $\alpha_k\geq \alpha_i$ for all $i$. For $i<k$,
	\begin{align*}
	&\ \alpha_k - \alpha_{i} =  \frac{(1 + 2\rho L_{i}) s_{k} - (1+2\rho L_{k}) s_{i}}{(1+2\rho L_{k})(1+2\rho L_{i})}=  \frac{(1 + 2\rho L_{k})(s_k-s_i) -2\rho s_k\sum_{j=i+1}^k\tilde{w}_j^2}{(1+2\rho L_{k})(1+2\rho L_{i})}\\
	&=  \frac{(1 + 2\rho L_{k})\sum_{j=i+1}^k\tilde{w}_j\tilde{a}_j -2\rho (1 + 2\rho L_{k})\beta\sum_{j=i+1}^k\tilde{w}_j^2}{(1+2\rho L_{k})(1+2\rho L_{i})}=  \frac{\sum_{j=i+1}^k \tilde{w}_j^2 (\tilde{a}_j/\tilde{w}_j -2\rho \beta)}{1+2\rho L_{i}}\geq 0.
	\end{align*}
	We can prove that $\alpha_k\geq \alpha_i$ for all $i>k$ in a similar way. Therefore, we have that $\beta=\alpha_k=\max_{1\leq i \leq n}\alpha_i=\bar{\alpha}$.
	
	Finally, since the solution to \eqref{eq:pos_problem} is unique, we have
	\[
	x(a)=x^*=(a - 2\rho\beta w)^+=(a - 2\rho\bar{\alpha} w)^+.
	\]
\end{proof}
With the results above, we now give the closed-form solution to ${\rm Prox}_{\rho\|w\circ\cdot\|_1^2}(a)$ for any $a \in \mathbb{R}^n$.
\begin{proposition}\label{prop:proxmapping}
	For given $\rho > 0$ and $a \in \mathbb{R}^n$, we have
	\[
	{\rm Prox}_{\rho\|w\circ\cdot\|_1^2}(a)={\rm sign}(a)\circ {\rm Prox}_{\rho\|w\circ\cdot\|_1^2}(|a|)={\rm sign}(a)\circ x(|a|),
	\]
	where $x(\cdot)$ is defined in \eqref{eq:pos_problem} and can be computed by Proposition \ref{prop:pos_prox}. (Hence  ${\rm Prox}_{\rho\|w\circ\cdot\|_1^2}(a)$ can be computed in $O(n\log n)$ operations.)
\end{proposition}
\begin{proof}
	Since $\|w\circ x\|_1^2$ is invariant to sign changes, the conclusion of this proposition hold.
\end{proof}

\subsection{The generalized Jacobian of ${\rm Prox}_{\rho\|w\circ\cdot\|_1^2}(\cdot)$ }
In order to design the SSN method to solve nonsmooth equations involving the exclusive lasso regularizer, it is critical for us to derive an explicit formula for some form of the generalized Jacobian of ${\rm Prox}_{\rho\|w\circ\cdot\|_1^2}(\cdot)$. Here, we derive a specific element in the set of the HS-Jacobian of ${\rm Prox}_{\rho\|w\circ\cdot\|_1^2}(\cdot)$ based on the quadratic programming (QP) reformulation of ${\rm Prox}_{\rho\|w\circ\cdot\|_1^2}(\cdot)$.

By Proposition \ref{prop:proxmapping}, we know that in order to get the generalized Jacobian of ${\rm Prox}_{\rho\|w\circ\cdot\|_1^2}(\cdot)$, we need to study the generalized Jacobian of $x(a)$ first. For any $a\in \mathbb{R}^n$, if we denote $ Q = I_n +  2\rho ww^T \in \mathbb{R}^{n\times n}$, \eqref{eq:pos_problem} can be equivalently written as
\begin{align}
x(a)=\arg\min_{x \in \mathbb{R}^n} \Big\{ \frac{1}{2} \langle x, Qx \rangle - \langle x, a \rangle \mid x\geq 0 \Big\}.\label{eq:QP}
\end{align}
Based on the above strongly convex QP, we can derive the HS-Jacobian of
$x(a)$ by applying the general results established in
\cite{han1997newton,li2018efficiently}, which will be described in the following paragraphs.

As one can see from \eqref{eq: pos_KKT} and the fact that $x(a)$ admits a unique solution, the corresponding dual multiplier $\mu$ also has a unique solution, which can be denoted as $\mu(a)$. The optimality conditions given in \eqref{eq: pos_KKT} can be equivalently given as
\begin{equation}\label{eq: KKT_QP}
Qx(a) - a + \mu(a) = 0,\;
\mu(a)^Tx(a) = 0,\;\mu(a) \leq 0, \; x(a)\geq 0.
\end{equation}
Denote the active set
\begin{equation}\label{eq: activeset}
I(a): = {\{i \in \{ 1,\ldots,n \} \mid x(a) = 0\}}.
\end{equation}
Now, we define a collection of index sets:
\[
{\cal K}(a): =  \{\;  K\subseteq \{1,\ldots,n\} \mid  {\rm supp}(\mu(a)) \subseteq K\subseteq I(a)\},
\]
where ${\rm supp}(\mu(a))$ denotes  the set of indices $i$ such that $\mu(a)_i \neq 0$. Note that the set ${\cal K}(a)$ is non-empty \cite{han1997newton}. Since the B-subdifferential $\partial_B x(a)$ is difficult to compute, according to the ideas in \cite{han1997newton,li2018efficiently}, we define the following multi-valued mapping $\partial_{\rm HS}x(a)$: $\mathbb{R}^{n} 	\rightrightarrows \mathbb{R}^{n\times n}$:
\begin{equation}\label{eq:HS_QP}
\partial_{\rm HS}x(a) := \left\{
P\in\mathbb{R}^{n\times n}\mid P = Q^{-1} - Q^{-1}
I_K^T\left( I_K Q^{-1}I_K^T\right)^{-1}I_K Q^{-1}, K\in{\cal K}(a)
\right\}
\end{equation}
as a computational replacement for $\partial_Bx(a)$, where $I_K$ is the matrix consisting of the rows of $I_n$, indexed by $K$. The set $\partial_{\rm HS}x(a)$ is known as the HS-Jacobian  of $x(\cdot)$ at $a$. The following proposition from \cite{li2018efficiently} provides some useful properties of $\partial_{\rm HS} x(a)$.
\begin{proposition}(\cite[Proposition 2]{li2018efficiently})\label{prop: QP_HS}
	For $a \in \mathbb{R}^n$, there exists a neighborhood $U$ of $a$ such that for any $a' \in U$, it holds that ${\cal K}(a') \subseteq {\cal K}(a)$, $\partial_{\rm HS}x(a') \subseteq \partial_{\rm HS}x(a)$. If ${\cal K}(a') \subseteq {\cal K}(a)$, then
	\[
	x(a') = x(a) + P(a' - a),\quad \forall P \in \partial_{\rm HS}x(a').
	\]
\end{proposition}

Based on the results in Proposition \ref{prop:proxmapping} and Proposition \ref{prop: QP_HS}, we can now compute a specific element in the HS-Jacobian $\partial_{\rm HS} {\rm Prox}_{\rho\|w\circ\cdot\|_1^2}(\cdot)$ at any $a \in \mathbb{R}^n$ as follows.
\begin{theorem}\label{thm: HS_jacobian}
	Given $a \in \mathbb{R}^n$, $\rho > 0$, the HS-Jacobian $\partial_{\rm HS} {\rm Prox}_{\rho\|w\circ\cdot\|_1^2}(a)$ can be given as
	\begin{align*}
	\partial_{\rm HS} {\rm Prox}_{\rho\|w\circ\cdot\|_1^2}(a)=\left\{
	\begin{array}{c}
	\Theta P\Theta\mid P \in \partial_{\rm HS}x(|a|)
	\end{array}
	\right\},
	\end{align*}
	where $\partial_{\rm HS}x(\cdot)$ is defined as in \eqref{eq:HS_QP} and $\Theta = {\rm Diag}({\rm sign}(a))$. Moreover, the matrix
	\begin{equation}\label{eq:M0}
	M_0 :=  \Theta  P_0  \Theta, \mbox{ with } P_0 =  Q^{-1} - Q^{-1}I_{I(|a|)}^T
	\left(
	I_{I(|a|)} Q^{-1}
	I_{I(|a|)}^T\right)^{-1}
	I_{I(|a|)}Q^{-1}.
	\end{equation}
	is an element in the HS-Jacobian $\partial_{\rm HS} {\rm Prox}_{\rho\|w\circ\cdot\|_1^2}(a)$.
\end{theorem}

For efficient implementation, we can use the result in the following proposition to compute $M_0$ in \eqref{eq:M0}, which is indeed a $0$-$1$ diagonal matrix plus a rank-one correction, and the proof can be found in the supplementary materials.
\begin{proposition}\label{compute_M}
	Define $\xi\in \mathbb{R}^{n}$ with $\xi_i = 0$ if $i\in I(|a|)$, and $\xi_i = 1$ otherwise, and $\Sigma={\rm Diag}(\xi)$. Denote $\tilde{w} =\Theta \Sigma w$, then $M_0$ defined in \eqref{eq:M0} can be computed as
	\begin{equation*}
	M_0 = \Sigma-
	\frac{2\rho}{1+2\rho (\tilde{w}^T\tilde{w})} \tilde{w}\tilde{w}^T.
	\end{equation*}
\end{proposition}

The next proposition shows the strong semismoothness of ${\rm Prox}_{\rho\|w\circ\cdot\|_1^2}(\cdot)$.
\begin{proposition}\label{prop: semismooth_of_l12}
	${\rm Prox}_{\rho\|w\circ\cdot\|_1^2}(\cdot)$ is strongly semismooth with respect to $\partial_{\rm HS} {\rm Prox}_{\rho\|w\circ\cdot\|_1^2}(\cdot)$
\end{proposition}
\begin{proof}
	As one can see, ${\rm Prox}_{\rho\|w\circ\cdot\|_1^2}(\cdot)$ is piecewise linear and Lipschitz continuous, thus it is directionally differentiable by \cite{facchinei2007finite}. From Proposition \ref{prop: QP_HS}, we know that there exists a neighborhood ${\cal U}$ of $a$ such that for all $a'\in {\cal U}$,
	\begin{align*}
	{\rm Prox}_{\rho\|w\circ\cdot\|_1^2}(a')-{\rm Prox}_{\rho\|w\circ\cdot\|_1^2}(a)-M(a'-a)=0,\quad \forall M\in \partial_{\rm HS} {\rm Prox}_{\rho\|w\circ\cdot\|_1^2}(a').
	\end{align*}
	Therefore, ${\rm Prox}_{\rho\|w\circ\cdot\|_1^2}(\cdot)$ is strongly semismooth with respect to $\partial_{\rm HS} {\rm Prox}_{\rho\|w\circ\cdot\|_1^2}(\cdot)$.
\end{proof}

\subsection{${\rm Prox}_p(\cdot)$ and its generalized Jacobian}
For $p(\cdot)=\lambda \Delta^{\mathcal{G},w}(\cdot)$, in order to explicitly give the closed-form solution to ${\rm Prox}_p(\cdot)$ and its generalized Jacobian, we need the following notations. For $i=1\cdots,l$, we define the linear mapping ${\cal P}_i:\mathbb{R}^n \rightarrow \mathbb{R}^{|g_i|}$ as ${\cal P}_i x=x_{g_i}$ for all $x\in \mathbb{R}^n$, and ${\cal P}=[{\cal P}_1;\cdots;{\cal P}_l]$. Let $n_i=\sum_{k=1}^i |g_k|$ and $n_0=1$. Denote $x^{(i)}$ as the sub-vector extracted from $x$ based on the index set $\{n_{i-1},n_{i-1}+1,\cdots,n_{i}\}$.

Based on these notations, the proximal mapping ${\rm Prox}_p(\cdot)$ can be computed as
\begin{align*}
{\rm Prox}_p(x)&={\cal P}^T
\arg\min_{y\in \mathbb{R}^n}\Big\{
\frac{1}{2}\|y-{\cal P}x\|^2+\lambda \sum_{i=1}^l \|({\cal P} w)^{(i)} \circ y^{(i)}\|_1^2\Big\}\\
&={\cal P}^T [{\rm Prox}_{\lambda\|({\cal P} w)^{(1)}\circ\cdot\|_1^2}(({\cal P} x)^{(1)});\cdots;{\rm Prox}_{\lambda\|({\cal P} w)^{(l)}\circ\cdot\|_1^2}(({\cal P} x)^{(l)})].
\end{align*}
In addition,
\begin{align*}
V={\cal P}^T {\rm Diag}(\Sigma_1-
\frac{2\lambda}{1+2\lambda (\tilde{w}_1^T\tilde{w}_1)} \tilde{w}_1\tilde{w}_1^T,\cdots,\Sigma_l-
\frac{2\lambda}{1+2\lambda (\tilde{w}_l^T\tilde{w}_l)} \tilde{w}_l\tilde{w}_l^T)   {\cal P}
\end{align*}
is an element in $\partial_{\rm HS}{\rm Prox}_p(x)$, where $\Sigma_i$, $\tilde{w}_i$ corresponds to $M_0$ for $\partial_{\rm HS} {\rm Prox}_{\lambda\|({\cal P} w)^{(i)}\circ\cdot\|_1^2}(({\cal P} x)^{(i)})$. The strong semismoothness of ${\rm Prox}_p(\cdot)$ w.r.t. $\partial_{\rm HS}{\rm Prox}_p(\cdot)$ follows naturally.

\section{Numerical experiments}
\label{sec:numerical}
In this section, we perform some numerical experiments to test our proposed PPDNA in solving the exclusive lasso model. For simplicity, we take the weight vector $w$ to be all ones. The exclusive lasso model can be described as
\begin{equation}\label{eq: exclusive-lasso}
\min_{x\in \mathbb{R}^n} \Big\{h(Ax) + \lambda \sum_{g\in \mathcal{G}} \| x_g\|_1^2\Big\},
\end{equation}
where $\mathcal{G} = \{g \mid g \subseteq \{1, 2, \dots, n\} \}$ is a disjoint partition of $\{1, 2, \dots, n\}$, $A\in \mathbb{R}^{m\times n}$ and $\lambda>0$. By taking $c = 0$, and $p(x) = \lambda \sum_{g\in \mathcal{G}} \| x_g\|_1^2$, we can reformulate \eqref{eq: exclusive-lasso} in the form of \eqref{eq: Convex-composite-program}. Thus, all the analyses in previous sections are applicable for the above model.  All our computational results are obtained by running {\sc Matlab} on a windows workstation (12-core, Intel Xeon E5-2680 @ 2.50GHz, 128G RAM).

In the numerical experiments, we mainly focus on two aspects. (1) We compare our proposed PPDNA for solving \eqref{eq: exclusive-lasso} to three popular state-of-the-art first-order frameworks, ILSA \cite{kong2014exclusive}, ADMM with step length $\kappa = 1.618$ \cite{fazel2013hankel} and APG with restart under the setting described in \cite{becker2011templates}.  To demonstrate the efficiency and scalability of PPDNA, we perform the time comparison on synthetic datasets from small to large scales. (2) We apply the exclusive lasso model \eqref{eq: exclusive-lasso} with least squares loss function to index ETF (exchange traded fund) in finance. The out-of-sample results show the superior performance of the exclusive lasso model in index tracking, comparing to the lasso and group lasso models.

We stop all the four algorithms by the following criterion  based on the relative KKT residual:
$$\eta_{\rm KKT} := \frac{\|x - {\rm Prox}_{p}(x - A^T\nabla h(Ax)\|}{1 + \|x\| + \|A^T\nabla h(Ax)\|} \leq \varepsilon,$$
where $\varepsilon > 0$ is a given tolerance, which is set to $10^{-6}$ in our experiments. We also terminate PPDNA when it reach the maximum iteration of $200$ and terminate ILSA, ADMM and APG when they reach the maximum iteration of $200000$ unless otherwise specified. In addition, we set the maximum computation time as $1$ hour.

\subsection{The regularized linear regression problem with synthetic data}
In this subsection, we test the efficiency of PPDNA for solving \eqref{eq: exclusive-lasso} with $h(y):=\sum_{i=1}^m (y_i-b_i)^2/2$, and compare it against ILSA, ADMM and APG on synthetic datasets. In this case, we apply the classical PPDNA, i.e. ${\cal M}_k\equiv I_n$ for all $k$. Here we focus on the time comparison among the algorithms. For the comparison of prediction error among the exclusive lasso, lasso and other linear regression models, we refer the readers to \cite{campbell2017within} for more details.

We adopt the design of synthetic datasets as described in \cite{campbell2017within}. We generate the synthetic data using the model $b = Ax ^* +\epsilon$, where $x ^*$ is the predefined true solution and $\epsilon \sim \mathcal{N}(0, I_m)$ is a random noise vector. Given the number of observations $m$, the number of groups $s$ and the number of features $p$ in each group, we generate each row of the matrix $A \in \mathbb{R}^{m \times sp}$ by sampling a vector from a multivariate normal distribution $\mathcal{N}(0, \Sigma)$, where $\Sigma$ is a Toeplitz covariance matrix with entries $\Sigma_{ij} = 0.9^{|i - j|}$ for features in the same group, and $\Sigma_{ij} = 0.3^{|i - j|}$ for features in different groups. For the ground-truth $x^{*}$, we randomly generate $10$ nonzero elements in each group with i.i.d values from the uniform distribution on $[0,10]$.

We mainly focus on feature selection by the exclusive lasso model in the high-dimensional settings. Hence, we fix $m$ to be $200$ and $s$ to be $20$, but vary the number of features $p$ in each group from $50$ to $1000$. That is, we vary the total number of features $n=sp$ from $1000$ to $20000$. To compare the robustness of different algorithms with respect to the parameter $\lambda$, we test all the algorithms under two different values of $\lambda$. The results are shown in Figure \ref{fig: time-comparison}, which demonstrate the superior performance of PPDNA, especially for large-scale instances, comparing to ILSA, ADMM and APG.

More results on higher dimensional cases are shown in Table \ref{tab: ls}. As one can see from Figure \ref{fig: time-comparison}, APG and ILSA are not efficient enough to solve large-scale instances, thus we only compare PPDNA with ADMM in these higher-dimensional cases. For the largest two instances in Table \ref{tab: ls}, PPDNA is able to solve the problems in two minutes whereas ADMM fails to solve them even after $1$ hour.

\begin{figure*}[!h]
	\begin{center}
		\includegraphics[width = 1\columnwidth]{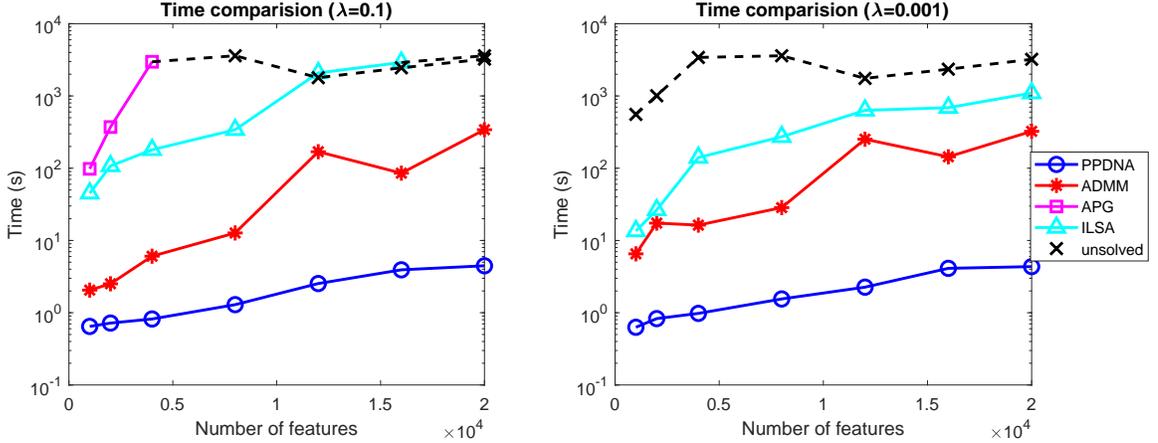}
		\setlength{\abovecaptionskip}{-10pt}
		\setlength{\belowcaptionskip}{-10pt}
		\caption{Time comparison among PPDNA, ILSA, ADMM and APG for linear regression on synthetic datasets. The black dash line with $'\times'$ indicates that the algorithm fails to solve the instance.}
		\label{fig: time-comparison}
	\end{center}
\end{figure*}

\begin{table*}[!h]\fontsize{8pt}{11pt}\selectfont
	\setlength{\abovecaptionskip}{0pt}
	\setlength{\belowcaptionskip}{0pt}
	\caption{Comparison between PPDNA and ADMM for linear regression on synthetic datasets. In the table, ``23(106)" means ``PPDNA iterations (total inner SSN iterations)". Time is in the format of (hours:minutes:seconds). Values in bold means that the algorithm fails to solve the instance to the required accuracy.}\label{tab: ls}
	\renewcommand\arraystretch{1.2}
	\centering
	\begin{tabular}{|c|c|c|c|c|}
		\hline
		& & iter &  $\eta_{\rm KKT}$ & time\\
		\hline
		Data $(m,s,p)$  & $\lambda$  & PPDNA $|$ ADMM & PPDNA $|$ ADMM & PPDNA $|$ ADMM \\
		\hline
		\multirow{2}*{\tabincell{c}{$(500,20,2000)$} }
		&1e-1 & 23(106) $|$ 23332  & 8.5e-7 $|$ 1.0e-6 & 0:00:27 $|$ 0:08:08  \\
		&1e-3 & 30(95) $|$ 167472  & 6.3e-7 $|$ \bf{1.5e-6}  & 0:00:24 $|$ 1:00:00 \\
		\hline
		\multirow{2}*{\tabincell{c}{$(500,20,3000)$} }
		&1e-1 & 23(97) $|$ 46226  & 3.9e-7 $|$ \bf{2.1e-6} & 0:00:36 $|$ 1:00:00  \\
		&1e-3 & 29(100) $|$ 50402  & 7.9e-7 $|$ \bf{9.0e-6} & 0:00:37 $|$ 1:00:01 \\
		\hline
		\multirow{2}*{\tabincell{c}{$(1000,20,2000)$} }
		&1e-1 & 21(132) $|$ 16208  & 5.0e-7 $|$ 1.0e-6 & 0:01:21 $|$ 0:09:03  \\
		&1e-3 & 28(160) $|$ 89242  & 7.8e-7 $|$ 1.0e-6 & 0:01:39 $|$ 0:50:41 \\
		\hline
		\multirow{2}*{\tabincell{c}{$(1000,20,4000)$} }
		&1e-1 & 22(107) $|$ 15644  & 7.1e-7 $|$ \bf{1.2e-5} & 0:01:44 $|$ 1:00:00  \\
		&1e-3 & 29(126) $|$ 15680  & 9.7e-7 $|$ \bf{3.6e-3} & 0:01:59 $|$ 1:00:01 \\
		\hline
	\end{tabular}
\end{table*}

\subsection{The regularized logistic regression problem with synthetic data}
In this subsection, we show the performance of PPDNA for solving the logistic regression model with the exclusive lasso regularizer. The logistic regression model could be formulated by taking
\[
h(y) = \mbox{$\sum_{i=1}^m$} \log(1 + \exp(-b_i y_i))
\]
in \eqref{eq: exclusive-lasso}, where $b\in \{-1,1\}^m$ is given. For robustness, we apply the preconditioned PPA with ${\cal M}_k\equiv I_n+\tau A^TA$ to solve this exclusive lasso model with $\tau=1/\lambda_{\max}(AA^T)$.

We use the same design of synthetic datasets described in the previous subsection, except for letting $b_i=1$ if $ Ax ^* +\epsilon\geq 0$, and $-1$ otherwise. As one can see in the previous subsection, APG and ILSA are very time-consuming when solving large-scale exclusive lasso problems compared to PPDNA and ADMM. Thus for logistic regression problems, we only compare PPDNA with ADMM. The description of the ADMM for solving the regularized logistic regression problem could be found in the supplementary materials. The numerical results are shown in Table \ref{tab: logistic}. Again, we can observe the superior performance of PPDNA against ADMM, and the performance gap is especially wide when the parameter $\lambda=10^{-5}$. For example, PPDNA is at least $50$ times faster than ADMM in solving the instance $(500,20,5000)$ with $\lambda=10^{-5}$.

\begin{table*}[!h]\fontsize{8pt}{11pt}\selectfont
	\caption{Time comparison between PPDNA and ADMM for logistic regression on synthetic datasets.}\label{tab: logistic}
	\renewcommand\arraystretch{1.2}
	\centering
	\begin{tabular}{|c|c|c|c|c|}
		\hline
		& & iter &  $\eta_{\rm KKT}$ & time\\
		\hline
		Data $(m,s,p)$  & $\lambda$  & PPDNA $|$ ADMM  & PPDNA $|$ ADMM & PPDNA $|$ ADMM \\
		\hline
		\multirow{3}*{\tabincell{c}{$(500,20,3000)$} }
		&1e-1 & 13(41) $|$ 1689  & 5.2e-7 $|$ 1.0e-6 & 0:00:14 $|$ 0:02:27  \\
		&1e-3 & 48(58) $|$ 5850  & 9.4e-7 $|$ 1.0e-6 & 0:00:20 $|$ 0:06:38 \\
		&1e-5 & 73(75) $|$ 17208  & 9.3e-7 $|$ 1.0e-6 & 0:00:27 $|$ 0:16:52 \\
		\hline
		\multirow{3}*{\tabincell{c}{$(500,20,5000)$} }
		&1e-1 & 12(45) $|$ 2167  & 2.6e-7 $|$ 1.0e-6 & 0:00:24 $|$ 0:04:45  \\
		&1e-3 & 37(51) $|$ 6187  & 2.1e-7 $|$ 1.0e-6 & 0:00:28 $|$ 0:10:17 \\
		&1e-5 & 67(68) $|$ 21584  & 8.9e-7 $|$ 1.0e-6 & 0:00:39 $|$ 0:33:54 \\
		\hline
		\multirow{3}*{\tabincell{c}{$(1000,20,5000)$} }
		&1e-1 & 13(46) $|$ 1186  & 2.9e-7 $|$ 1.0e-6 & 0:00:52 $|$ 0:06:12  \\
		&1e-3 & 47(62) $|$ 5593  & 6.6e-7 $|$ 1.0e-6 & 0:01:10 $|$ 0:22:45 \\
		&1e-5 & 66(68) $|$ 17829  & 9.9e-7 $|$ \bf{2.2e-6} & 0:01:23 $|$ 1:00:00 \\
		\hline
		\multirow{3}*{\tabincell{c}{$(1000,20,8000)$} }
		&1e-1 & 13(50) $|$ 1947  & 9.1e-8 $|$ 1.0e-6 & 0:01:24 $|$ 0:14:02  \\
		&1e-3 & 57(69) $|$ 6991  & 9.2e-7 $|$ 1.0e-6 & 0:02:00 $|$ 0:39:01 \\
		&1e-5 & 89(90) $|$ 10519  & 9.7e-7 $|$ \bf{1.4e-5} & 0:02:40 $|$ 1:00:00 \\
		\hline
		\multirow{3}*{\tabincell{c}{$(2000,20,10000)$} }
		&1e-1 & 11(48) $|$ 1625  & 7.3e-7 $|$ 1.0e-6 & 0:03:23 $|$ 0:33:02  \\
		&1e-3 & 62(72) $|$ 3522  & 6.0e-7 $|$ \bf{5.8e-5} & 0:05:17 $|$ 1:00:05 \\
		&1e-5 & 79(80) $|$ 4415  & 9.9e-7 $|$ \bf{2.5e-4} & 0:06:27 $|$ 1:00:05 \\
		\hline
	\end{tabular}
\end{table*}

\subsection{Application: index exchange-traded fund (index ETF)}
In this subsection, we apply the exclusive lasso model in a real application in finance. Consider the portfolio selection problem where a fund manager wants to select a small subset of stocks (to minimize transaction costs and business analyses) to track a target time series such as the S\&P 500 index. Furthermore, in order to diversify the risks, the portfolio is required to span across all sectors. Such an application naturally leads us to consider the exclusive lasso model.

\begin{figure*}[!h]
	\begin{center}
		\includegraphics[width = 0.4\columnwidth]{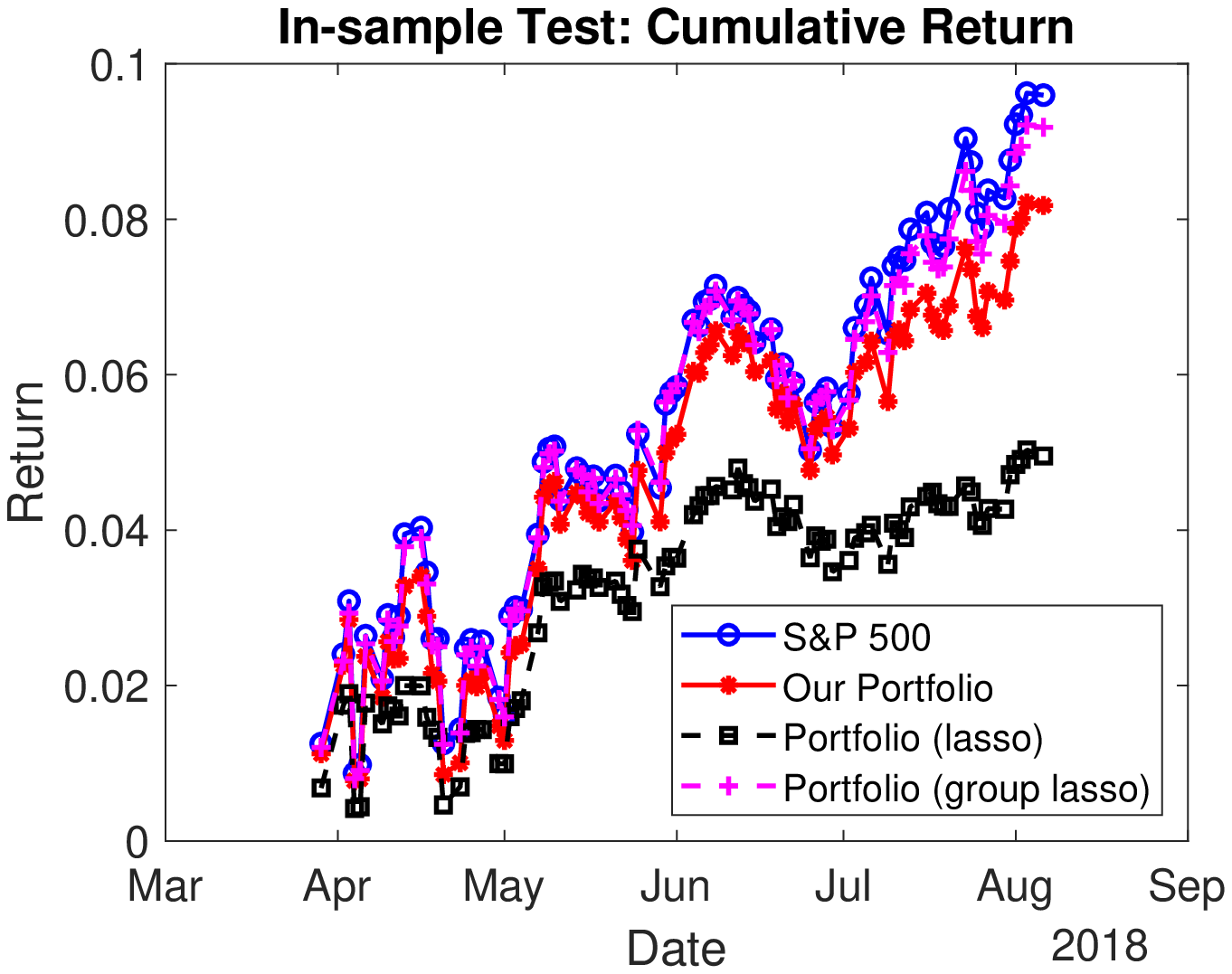}
		\quad
		\includegraphics[width = 0.43\columnwidth]{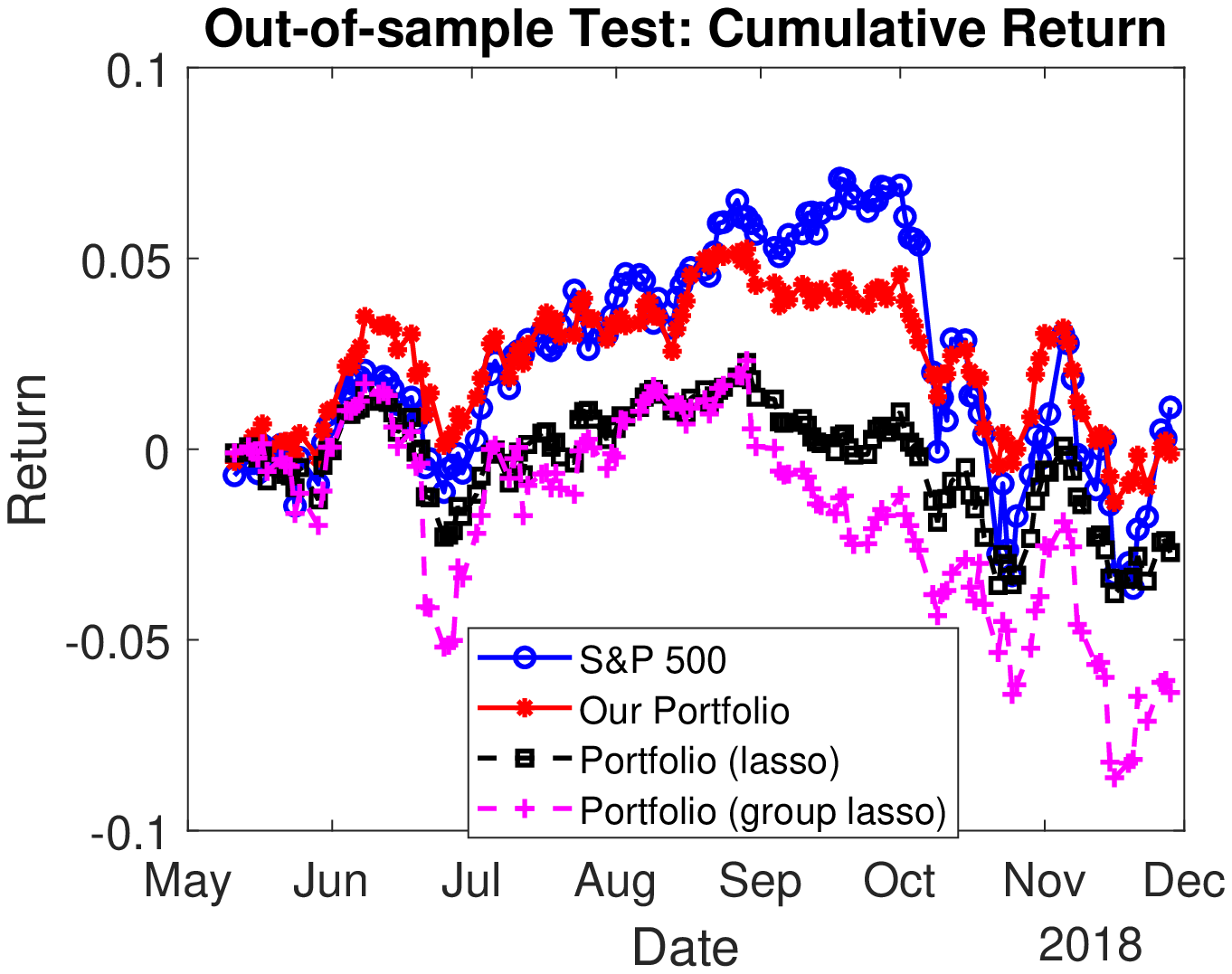}
		\setlength{\abovecaptionskip}{0pt}
		\setlength{\belowcaptionskip}{0pt}
		\caption{In-sample and out-of-sample performance of the exclusive lasso, the group lasso and the lasso model for index tracking of S\&P 500.}
		\label{fig: partial-index-tracking}
	\end{center}
\end{figure*}

\begin{figure}[!h]
	\begin{center}
		\includegraphics[width = 0.47\columnwidth]{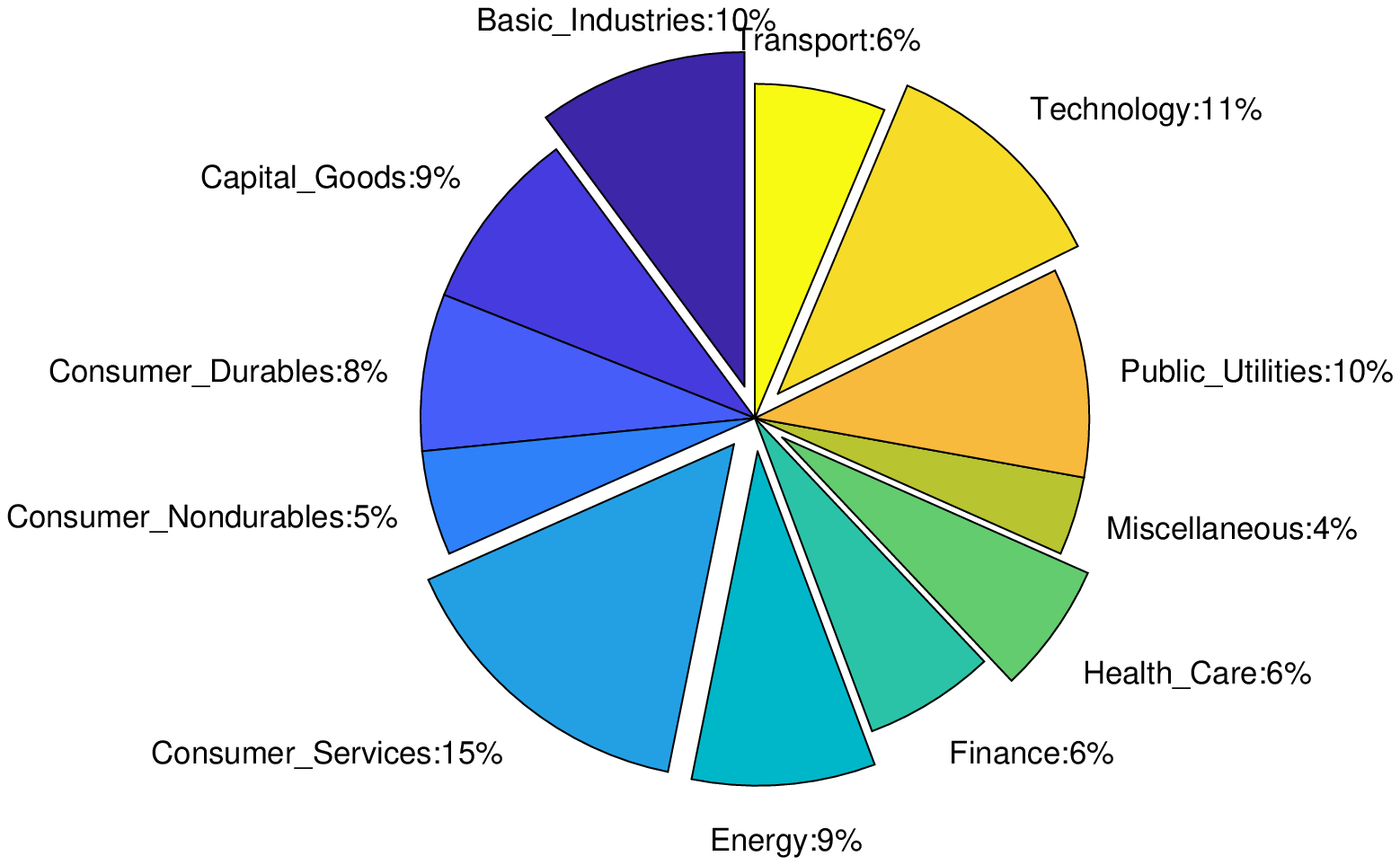}
		\quad
		\includegraphics[width = 0.47\columnwidth]{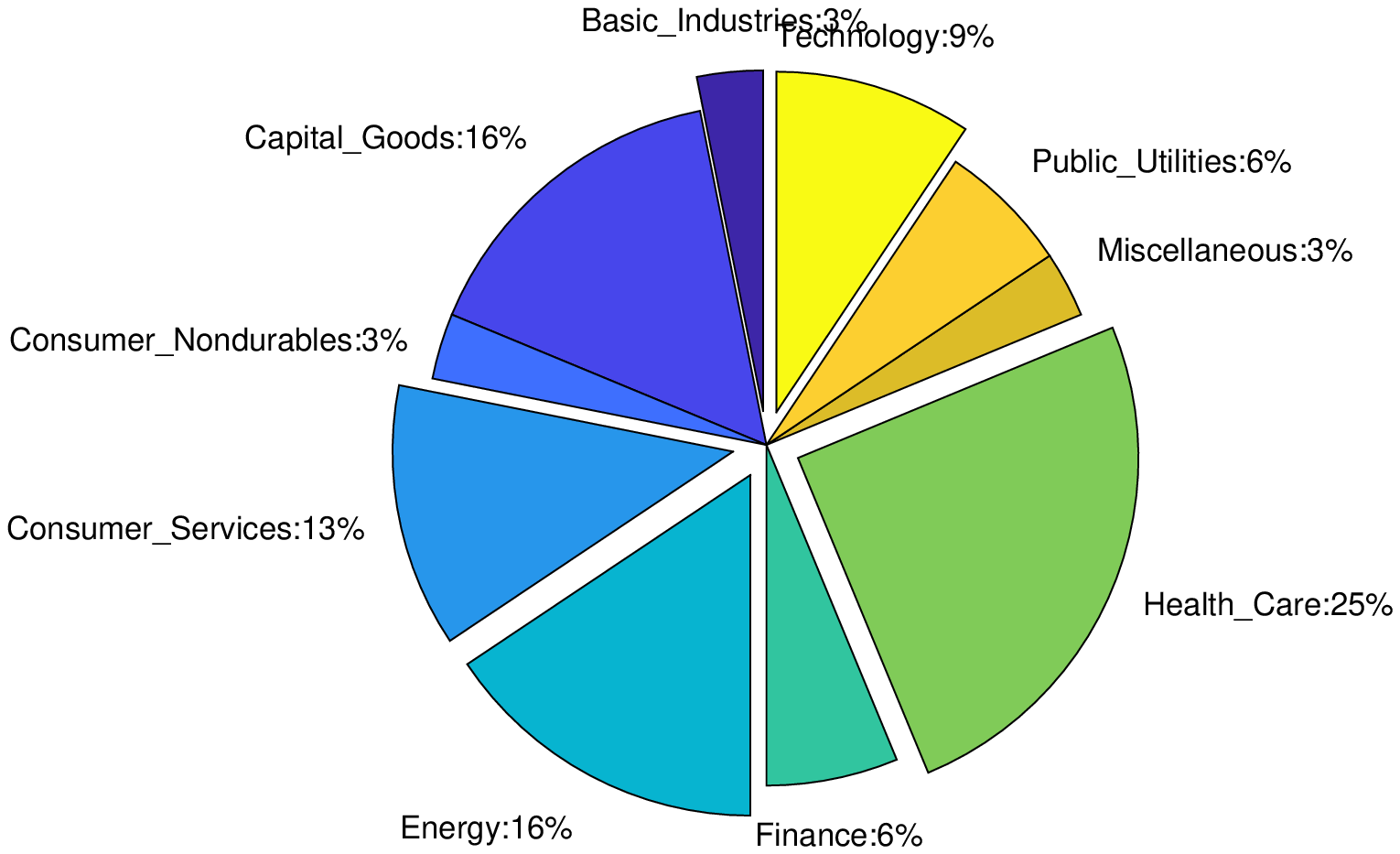}
		\vskip 0.2in
		\includegraphics[width = 0.47\columnwidth]{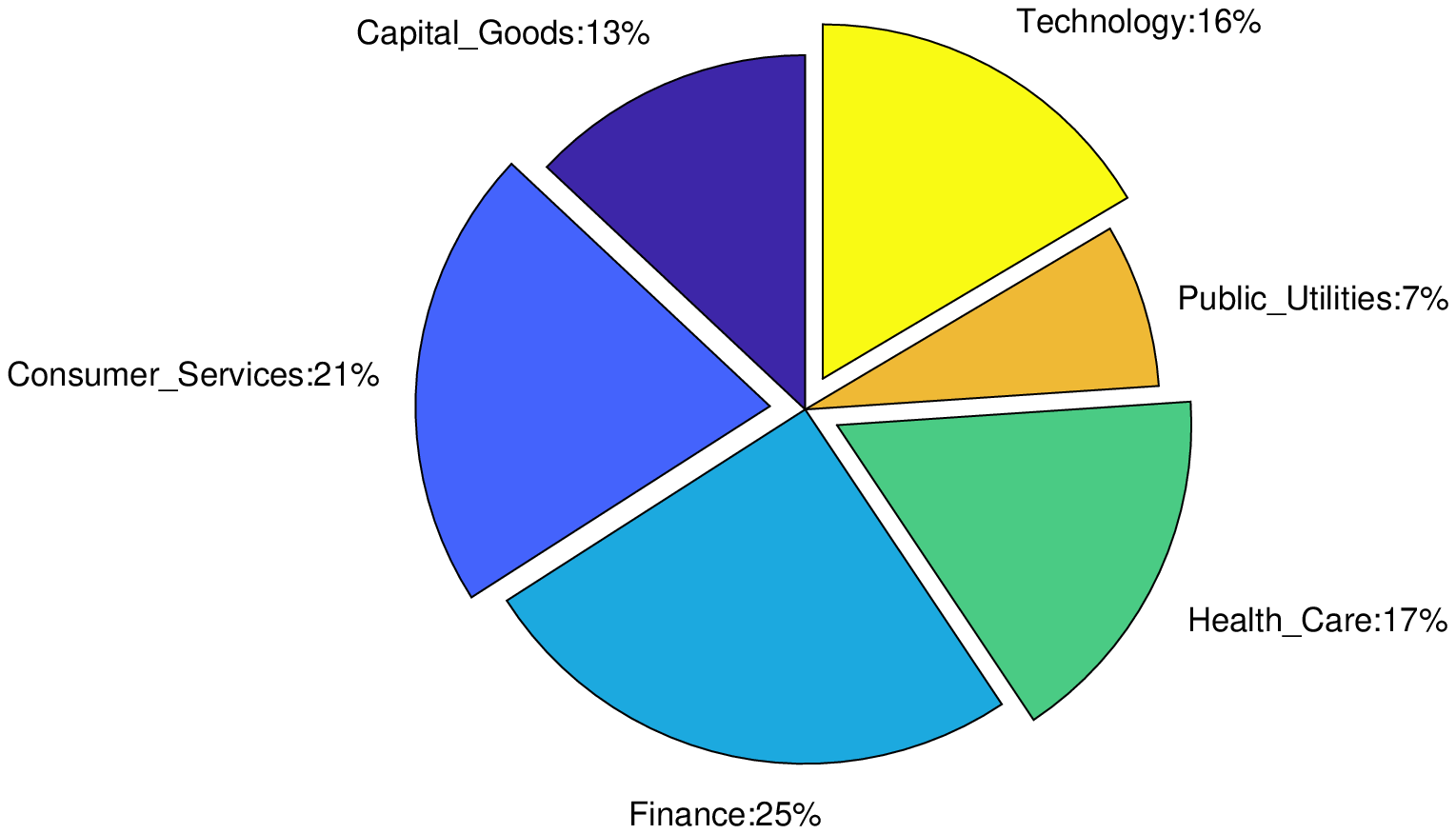}
		\setlength{\abovecaptionskip}{4pt}
		\setlength{\belowcaptionskip}{0pt}
		\caption{Percentage of selected stocks by sectors. Top left: exclusive lasso model. Top right: lasso model. Bottom: group lasso model.}
		\label{fig: pie-fig}
	\end{center}
\end{figure}

In our experiments, we download all the stock price data in the US market between 2018-01-01 and 2018-12-31 (251 trading days) from Yahoo finance \cite{yahoo-finance2018}. We drop the stock if more than 10\% of its price data is missing. After that, we get 3074 stocks in our stock universe. For the remaining stocks,  we handle the missing data via the common practice of forward interpolation. We then compute the daily return and get the historical {return} matrix $R \in \mathbb{R}^{250 \times 3074}$. We try to build a portfolio to track the S\&P 500  index. Let $y \in \mathbb{R}^{250}$ be the daily return of the S\&P 500 index in 2018. Since there are 12 sectors in the US market (e.g., finance, healthcare, technology, etc.), we have a natural group partition for our stock universe as $\mathcal{G}_{\rm US} = \{g_1, g_2, \dots, g_{12}\}$, where $g_i$ is the index set for stocks in the $i$-th sector.

To test the performance of the exclusive lasso model in index tracking, we use the rolling window method to test the in-sample and out-of-sample performance of the model. We use the historical data in the last 90 trading days to estimate a portfolio vector via the model for the future 10 days. More specifically, at day T, we solve the following problem\footnote{We explain why we could drop the constraints: $x \geq 0$ and $\sum_{i} x_i = 1$ in the supplementary materials.}:
\begin{align*}
x^*_T = \arg\min_{x} \frac{1}{2}\|  R_T x- y_T \|_2^2 + \lambda_T \mbox{$\sum_{g\in \mathcal{G}_{\rm US}}$} \|x_g\|_1^2,
\end{align*}
where $R_T$, $y_T$ are the daily return matrix of all stocks and daily return vector of S\&P 500 index in the last 90 trading days prior to day T, respectively. We select the parameter $\lambda_T$ using 9-folds cross validation. After we get the estimated portfolio vector $x^*_T$, we invest in the market based on it for the next 10 trading days. The in-sample and out-of-sample performance of the exclusive lasso model, the lasso model and the group lasso model is shown in Figure \ref{fig: partial-index-tracking}.

We plot the percentage of stocks from each sector in the portfolio obtained from the three tested models in Figure \ref{fig: pie-fig}. The result shows that our exclusive  lasso model can select stocks from all the 12 sectors, but the lasso model selects stocks only from 10 sectors and the group lasso model selects stocks only from 6 sectors in the universe. Moreover, the out-of-sample performance of the exclusive lasso model is visibly better than those corresponding to the lasso and group lasso models.

\section{Conclusion}
\label{sec:conclusion}
In this paper, we provide a rigorous proof for the closed-form solution to the proximal mapping of the exclusive lasso regularizer and derive its corresponding HS-Jacobian. Based on these theoretical results, we design a highly efficient and scalable second-order type algorithm (PPDNA) to solve the exclusive lasso model. Numerical results show that our PPDNA is far more efficient and robust  than popular first-order methods such as ADMM and APG methods. We apply the exclusive lasso model in an index ETF portfolio selection problem, and demonstrate that it can  achieve better out-of-sample performance comparing to the lasso model and group lasso model.

\newpage
\appendix
\section*{Appendices}
\section{Definition of semismoothness}
The concept of semismoothness is as follows, which can be found in \cite{kummer1988newton,mifflin1977semismooth,qi1993nonsmooth,sun2002semismooth}.

\begin{definition}(Semismoothness)
	For a given open set $\mathcal{O} \subseteq \mathbb{R}^n$, let $F:\mathcal{O} \rightarrow \mathbb{R}^m$ be a locally Lipschitz continuous function and  $\mathcal{G}: \mathcal{O} \rightrightarrows \mathbb{R}^{m \times n}$ be a nonempty compact valued upper-semicontinuous multifunction. $F$ is said to be semismooth at $x \in \mathcal{O}$ with respect to the multifunction $\mathcal{G}$ if $F$ is directionally differentiable at $x$ and for any $V \in \mathcal{G}(x+ \Delta x)$ with $\Delta x\rightarrow 0$,
	\[
	F(x+\Delta x) - F(x) - V\Delta x = o(\|\Delta x\|).
	\]
	$F$ is said to be strongly semismooth at $x \in \mathcal{O}$ with respect to $\mathcal{G}$ if it is semismooth at $x$ with respect to $\mathcal{G}$ and
	\[
	F(x+\Delta x) - F(x) - V\Delta x = O(\|\Delta x\|^2).
	\]
	$F$ is said to be semismooth (respectively, strongly semismooth) on $\mathcal{O}$ with respect to $\mathcal{G}$ if it is semismooth (respectively, strongly semismooth) everywhere in $\mathcal{O}$ with respect to $\mathcal{G}$.
\end{definition}

\section{Proof of Proposition \ref{prop:stop_of_ATA}}
By noting that 
\[
f_k(x)=f(x)+\frac{1}{2\sigma_k}\|x-x^k\|_{{\cal M}_k}^2,
\]
we know from \cite[Exercise 8.8]{rockafellar2009variational} that
\[
\partial f_k(x)=\partial f(x)+\frac{1}{\sigma_k} {\cal M}_k(x-x^k).
\]
Since ${\cal P}_k(x^k)=\arg\min f_k(x)$, we have that $0\in \partial f_k({\cal P}_k(x^k))$, which means there exists $v\in \partial f({\cal P}_k(x^k))$ such that 
\[
0=v+\frac{1}{\sigma_k} {\cal M}_k({\cal P}_k(x^k)-x^k).
\]
Since $f_k({\cal P}_k(x^k))=\inf f_k$, it holds that
\begin{align*}
f_k(x^{k+1})-\inf f_k&=f(x^{k+1})-f({\cal P}_k(x^k))+\frac{1}{2\sigma_k}\|x^{k+1}-x^k\|_{{\cal M}_k}^2-\frac{1}{2\sigma_k}\|{\cal P}_k(x^k)-x^k\|_{{\cal M}_k}^2\\
&=f(x^{k+1})-f({\cal P}_k(x^k))+\frac{1}{2\sigma_k}\langle x^{k+1}+{\cal P}_k(x^k)-2x_k,x^{k+1}-{\cal P}_k(x^k)\rangle_{{\cal M}_k}\\
&\geq \langle v,x^{k+1}-{\cal P}_k(x^k)\rangle+\frac{1}{2\sigma_k}\langle x^{k+1}+{\cal P}_k(x^k)-2x_k,x^{k+1}-{\cal P}_k(x^k)\rangle_{{\cal M}_k}\\
&=\frac{1}{2\sigma_k}\|x^{k+1}-{\cal P}_k(x^k)\|_{{\cal M}_k}^2.
\end{align*}
By the strongly duality, we know that $\inf f_k=\sup \psi_k$, thus
\[
\frac{1}{2\sigma_k}\|x^{k+1}-{\cal P}_k(x^k)\|_{{\cal M}_k}^2\leq f_k(x^{k+1})-\inf f_k=f_k(x^{k+1})-\sup \psi_k\leq f_k(x^{k+1})-\psi_k(u^{k+1}).
\]
Therefore, the stopping criteria \eqref{eq:stopA_pre_ppa} and \eqref{eq:stopB_pre_ppa} can be achieved by:
\begin{align}
f_k(x^{k+1})-\psi_k(u^{k+1})&\leq \frac{\epsilon_k^2}{2\sigma_k},\quad \epsilon_k \geq 0,\quad \sum_{k=0}^{\infty}\epsilon_k <\infty,\tag{A''}\\
f_k(x^{k+1})-\psi_k(u^{k+1})&\leq \frac{\delta_k^2}{2\sigma_k}\|x^{k+1}-x^k\|_{{\cal M}_k}^2,\quad 0\leq \delta_k < 1,\quad \sum_{k=0}^{\infty}\delta_k <\infty.\tag{B''}
\end{align}

\section{Proof of Proposition \ref{compute_M}} Let $\Xi=I_n-\Sigma$. It can be proved that 
\begin{align*}
I_{I(|a|)}^T
\left(
I_{I(|a|)} Q^{-1}
I_{I(|a|)}^T\right)^{-1}
I_{I(|a|)}=(\Xi Q^{-1}\Xi)^{\dagger}=\Xi (\Xi Q^{-1}\Xi)^{\dagger}\Xi,
\end{align*}
where the last inequality follows from that $\Xi$ is a $0$-$1$ diagonal matrix. Then by \cite[Proposition 3]{li2018efficiently}, we can see that
\begin{align*}
P_0 &=  Q^{-1} - Q^{-1}I_{I(|a|)}^T
\left(
I_{I(|a|)} Q^{-1}
I_{I(|a|)}^T\right)^{-1}
I_{I(|a|)}Q^{-1}\\
&=Q^{-1} - Q^{-1}\Xi (\Xi Q^{-1}\Xi)^{\dagger}\Xi Q^{-1}\\
&= (\Sigma Q \Sigma)^{\dagger}.
\end{align*}
Since $Q = I_n +  2\rho ww^T \in \mathbb{R}^{n\times n}$, denoting $\hat{w}=\Sigma w$,
we have that
\begin{align*}
P_0 &= (\Sigma Q \Sigma)^{\dagger}
= (\Sigma + 2\rho \hat{w}\hat{w}^T)^{\dagger}=\Sigma-
\frac{2\rho}{1+2\rho (\hat{w}^T\hat{w})} \hat{w}\hat{w}^T.
\end{align*}
Note that $\Theta = {\rm Diag}({\rm sign}(a))$ is a diagonal matrix with its diagonal elements being $1$ or $-1$, 
\begin{equation*}
M_0 =  \Theta  \Omega  \Theta = \Theta (\Sigma-
\frac{2\rho}{1+2\rho (\hat{w}^T\hat{w})} \hat{w}\hat{w}^T) \Theta = \Sigma-
\frac{2\rho}{1+2\rho (\hat{w}^T\hat{w})} \tilde{w}\tilde{w}^T= \Sigma-
\frac{2\rho}{1+2\rho (\tilde{w}^T\tilde{w})} \tilde{w}\tilde{w}^T,
\end{equation*}
where $\tilde{w} = \Theta \hat{w}= \Theta\Sigma w$.

\section{ADMM for solving the regularized logistic regression problem}
The minimization form of the dual of \eqref{eq: Convex-composite-program} is given as 
\begin{align}
\min_{w,u\in \mathbb{R}^m,v\in \mathbb{R}^n} \{ h^*(w) + p^*(v)\mid {\cal A}^* u+v-c=0,w-u=0\}.\label{eq:dual_org}
\end{align}
The augmented Lagrangian function associated with \eqref{eq:dual_org} is 
\begin{align*}
{\cal L}_{\sigma}(w,u,v;x,y)=& h^*(w) + p^*(v)-\langle x,{\cal A}^* u+v-c\rangle - \langle y,w-u\rangle \\&+\frac{\sigma}{2}\|{\cal A}^* u+v-c\|^2+\frac{\sigma}{2}\|w-u\|^2.
\end{align*}
The alternating direction method of multipliers (ADMM) for solving \eqref{eq: Convex-composite-program} and \eqref{eq:dual_org} could be described as 
\begin{subequations}
	\begin{align}
	&u^{k+1}= \arg\min_u {\cal L}_{\sigma}(w^k,u,v^k;x^k,y^k),
	\label{eq:admm_u} \\
	&(w^{k+1},v^{k+1}) =\arg\min_{w,v} {\cal L}_{\sigma}(w,u^{k+1},v;x^k,y^k), \label{eq:admm_v}\\
	&x^{k+1} =x^k-\kappa\sigma({\cal A}^*u^{k+1}+v^{k+1}-c),\quad y^{k+1}=y^k-\kappa\sigma(w^{k+1}-u^{k+1}),\label{eq:admm_xy}
	\end{align}
\end{subequations}
where the step length $\kappa=1.618$ and $\sigma>0$ is a given parameter. For the subproblem \eqref{eq:admm_v}, $w$ and $v$ can be computed simultaneously as
\begin{align*}
w^{k+1}& ={\rm Prox}_{h^*/\sigma}(u^{k+1}+y^k/\sigma)\\
&=(u^{k+1}+y^k/\sigma)-\frac{1}{\sigma}{\rm Prox}_{\sigma h}(\sigma u^{k+1}+y^k),\\
v^{k+1}& ={\rm Prox}_{p^*/\sigma}(-{\cal A}^*u^{k+1}+c+x^k/\sigma)\\
&=(-{\cal A}^*u^{k+1}+c+x^k/\sigma)-\frac{1}{\sigma}{\rm Prox}_{\sigma p}(-\sigma {\cal A}^*u^{k+1}+\sigma c+x^k),
\end{align*}
where the Moreau identity ${\rm Prox}_{tp}(x)+t{\rm Prox}_{f^*/t}(x/t)=x$ is used.
For the subproblem \eqref{eq:admm_u}, the optimality condition is 
\begin{align*}
(I_m+{\cal A}{\cal A}^*)u = {\cal A}(c+x^k/\sigma-v^{k})+(w^k-y^k/\sigma).
\end{align*}
One can solve this linear system directly or use an iterative solver such as the preconditioned conjugate gradient method.

\section{The explanation of the model of index ETF}
Here we explain why we can drop the simplex constraint $x \geq 0$, $\sum_{i} x_i = 1$ in the index ETF application. We assume that we can short stocks in the market, which means we can drop the nonnegative constraint $x \geq 0$. Furthermore, we assume the interest rate is $r_C$. Then for a given return vector $r \in \mathbb{R}^n$ of n stocks and a portfolio vector $x^*$, the return of the whole investment is given by
\begin{align*}
r^Tx^* + (1-\sum_{i=1}^n x^*_i)r_C
= \sum_{i=1}^n (r_i - r_C)x_i^* + r_C.
\end{align*}
Then, if we assume $r_C = 0$, or just set
\[r_{\rm new} = r- r_C, \quad y_{\rm new} = y - r_C.\]
We could drop the constraint $\sum_{i=1}^n x_i = 1$ in the index ETF model.


\begin{thebibliography}{10}
	
	\bibitem{yahoo-finance2018}
	{\em Yahoo {F}inance: https://finance.yahoo.com}.
	
	\bibitem{bauschke1999strong}
	{\sc H.~H. Bauschke, J.~M. Borwein, and W.~Li}, {\em Strong conical hull
		intersection property, bounded linear regularity, {J}ameson's property ({G}),
		and error bounds in convex optimization}, Mathematical Programming, 86
	(1999), pp.~135--160.
	
	\bibitem{beck2009fast}
	{\sc A.~Beck and M.~Teboulle}, {\em A fast iterative shrinkage-thresholding
		algorithm for linear inverse problems}, SIAM J. on Imaging Sciences, 2
	(2009), pp.~183--202.
	
	\bibitem{becker2011templates}
	{\sc S.~R. Becker, E.~J. Cand{\`e}s, and M.~C. Grant}, {\em Templates for
		convex cone problems with applications to sparse signal recovery},
	Mathematical Programming Computation, 3 (2011), p.~165.
	
	\bibitem{campbell2017within}
	{\sc F.~Campbell and G.~I. Allen}, {\em Within group variable selection through
		the exclusive lasso}, Electronic J. of Statistics, 11 (2017), pp.~4220--4257.
	
	\bibitem{cui2016asymptotic}
	{\sc Y.~Cui, D.~F. Sun, and K.-C. Toh}, {\em On the asymptotic superlinear
		convergence of the augmented {L}agrangian method for semidefinite programming
		with multiple solutions}, arXiv preprint arXiv:1610.00875,  (2016).
	
	\bibitem{dong2009extension}
	{\sc Y.~Dong}, {\em An extension of {L}uque's growth condition}, Applied
	Mathematics Letters, 22 (2009), pp.~1390--1393.
	
	\bibitem{dontchev2009implicit}
	{\sc A.~L. Dontchev and R.~T. Rockafellar}, {\em Implicit {F}unctions and
		{S}olution {M}appings}, Springer Monographs in Mathematics. Springer, 208
	(2009).
	
	\bibitem{eckstein1992douglas}
	{\sc J.~Eckstein and D.~P. Bertsekas}, {\em On the {D}ouglas-{R}achford
		splitting method and the proximal point algorithm for maximal monotone
		operators}, Mathematical Programming, 55 (1992), pp.~293--318.
	
	\bibitem{efron2004least}
	{\sc B.~Efron, T.~Hastie, I.~Johnstone, and R.~Tibshirani}, {\em Least angle
		regression}, The Annals of Statistics, 32 (2004), pp.~407--499.
	
	\bibitem{facchinei2007finite}
	{\sc F.~Facchinei and J.-S. Pang}, {\em Finite-{D}imensional {V}ariational
		{I}nequalities and {C}omplementarity {P}roblems}, Springer Science \&
	Business Media, 2007.
	
	\bibitem{fazel2013hankel}
	{\sc M.~Fazel, T.~K. Pong, D.~F. Sun, and P.~Tseng}, {\em Hankel matrix rank
		minimization with applications to system identification and realization},
	SIAM J. on Matrix Analysis and Applications, 34 (2013), pp.~946--977.
	
	\bibitem{glowinski1975approximation}
	{\sc R.~Glowinski and A.~Marroco}, {\em Sur l'approximation, par
		{\'e}l{\'e}ments finis d'ordre un, et la r{\'e}solution, par
		p{\'e}nalisation-dualit{\'e} d'une classe de probl{\`e}mes de dirichlet non
		lin{\'e}aires}, Revue fran{\c{c}}aise d'automatique, informatique, recherche
	op{\'e}rationnelle. Analyse num{\'e}rique, 9 (1975), pp.~41--76.
	
	\bibitem{han1997newton}
	{\sc J.~Han and D.~F. Sun}, {\em Newton and quasi-{N}ewton methods for normal
		maps with polyhedral sets}, J. of optimization Theory and Applications, 94
	(1997), pp.~659--676.
	
	\bibitem{kong2014exclusive}
	{\sc D.~Kong, R.~Fujimaki, J.~Liu, F.~Nie, and C.~Ding}, {\em Exclusive feature
		learning on arbitrary structures via $l_{1,2}$-norm}, in Advances in Neural
	Information Processing Systems, 2014, pp.~1655--1663.
	
	\bibitem{kowalski2009sparse}
	{\sc M.~Kowalski}, {\em Sparse regression using mixed norms}, Applied and
	Computational Harmonic Analysis, 27 (2009), pp.~303--324.
	
	\bibitem{kummer1988newton}
	{\sc B.~Kummer}, {\em Newton's method for non-differentiable functions},
	Advances in Mathematical Optimization, 45 (1988), pp.~114--125.
	
	\bibitem{li2018highly}
	{\sc X.~Li, D.~F. Sun, and K.-C. Toh}, {\em A highly efficient semismooth
		{N}ewton augmented {L}agrangian method for solving {L}asso problems}, SIAM J.
	on Optimization, 28 (2018), pp.~433--458.
	
	\bibitem{li2018efficiently}
	{\sc X.~Li, D.~F. Sun, and K.-C. Toh}, {\em On efficiently solving the
		subproblems of a level-set method for fused lasso problems}, SIAM J. on
	Optimization, 28 (2018), pp.~1842--1866.
	
	\bibitem{li2019asymptotically}
	{\sc X.~Li, D.~F. Sun, and K.-C. Toh}, {\em An asymptotically superlinearly
		convergent semismooth {N}ewton augmented {L}agrangian method for {L}inear
		{P}rogramming}, arXiv preprint arXiv:1903.09546,  (2019).
	
	\bibitem{liu2012implementable}
	{\sc Y.-J. Liu, D.~F. Sun, and K.-C. Toh}, {\em An implementable proximal point
		algorithmic framework for nuclear norm minimization}, Mathematical
	Programming, 133 (2012), pp.~399--436.
	
	\bibitem{luque1984asymptotic}
	{\sc F.~J. Luque}, {\em Asymptotic convergence analysis of the proximal point
		algorithm}, SIAM J. on Control and Optimization, 22 (1984), pp.~277--293.
	
	\bibitem{moreau1965proximite}
	{\sc J.-J. Moreau}, {\em Proximit{\'e} et dualit{\'e} dans un espace
		hilbertien}, Bulletin de la Soci{\'e}t{\'e} math{\'e}matique de France, 93
	(1965), pp.~273--299.
	
	\bibitem{nesterov2013gradient}
	{\sc Y.~Nesterov}, {\em Gradient methods for minimizing composite functions},
	Mathematical Programming, 140 (2013), pp.~125--161.
	
	\bibitem{qi1993nonsmooth}
	{\sc L.~Qi and J.~Sun}, {\em A nonsmooth version of {Newton's} method},
	Mathematical Programming, 58 (1993), pp.~353--367.
	
	\bibitem{robinson1981some}
	{\sc S.~M. Robinson}, {\em Some continuity properties of polyhedral
		multifunctions}, in Mathematical Programming at Oberwolfach, Springer, 1981,
	pp.~206--214.
	
	\bibitem{rockafellar1970convex}
	{\sc R.~T. Rockafellar}, {\em Convex {A}nalysis}, Princeton University Press,
	1970.
	
	\bibitem{rockafellar1976augmented}
	{\sc R.~T. Rockafellar}, {\em Augmented {Lagrangians} and applications of the
		proximal point algorithm in convex programming}, Mathematics of Operations
	Research, 1 (1976), pp.~97--116.
	
	\bibitem{rockafellar1976monotone}
	{\sc R.~T. Rockafellar}, {\em Monotone operators and the proximal point
		algorithm}, SIAM J. on Control and Optimization, 14 (1976), pp.~877--898.
	
	\bibitem{rockafellar2009variational}
	{\sc R.~T. Rockafellar and R.~J.-B. Wets}, {\em Variational {A}nalysis},
	vol.~317, Springer Science \& Business Media, 2009.
	
	\bibitem{sun1986monotropic}
	{\sc J.~Sun}, {\em On monotropic piecewise quadratic programming}, PhD thesis,
	University of Washington, 1986.
	
	\bibitem{tibshirani1996regression}
	{\sc R.~Tibshirani}, {\em Regression shrinkage and selection via the lasso}, J.
	of the Royal Statistical Society: Series B (Methodological), 58 (1996),
	pp.~267--288.
	
	\bibitem{yamada2017localized}
	{\sc M.~Yamada, T.~Koh, T.~Iwata, J.~Shawe-Taylor, and S.~Kaski}, {\em
		Localized {L}asso for high-dimensional regression}, in Artificial
	Intelligence and Statistics, 2017, pp.~325--333.
	
	\bibitem{yuan2006model}
	{\sc M.~Yuan and Y.~Lin}, {\em Model selection and estimation in regression
		with grouped variables}, J. of the Royal Statistical Society: Series B
	(Statistical Methodology), 68 (2006), pp.~49--67.
	
	\bibitem{zhang2016robust}
	{\sc T.~Zhang, B.~Ghanem, S.~Liu, C.~Xu, and N.~Ahuja}, {\em Robust visual
		tracking via exclusive context modeling}, IEEE Transactions on Cybernetics,
	46 (2016), pp.~51--63.
	
	\bibitem{zhang2018efficient}
	{\sc Y.~Zhang, N.~Zhang, D.~F. Sun, and K.-C. Toh}, {\em An efficient {H}essian
		based algorithm for solving large-scale sparse group {L}asso problems},
	Mathematical Programming,  (2018), pp.~1--41.
	
	\bibitem{zhao2010newton}
	{\sc X.-Y. Zhao, D.~F. Sun, and K.-C. Toh}, {\em {A {N}ewton-{CG} augmented
			Lagrangian method for semidefinite programming}}, SIAM J. on Optimization, 20
	(2010), pp.~1737--1765.
	
	\bibitem{zhou2010exclusive}
	{\sc Y.~Zhou, R.~Jin, and S.~C.-H. Hoi}, {\em Exclusive lasso for multi-task
		feature selection}, in Proceedings of the Thirteenth International Conference
	on Artificial Intelligence and Statistics, 2010, pp.~988--995.
	
	\bibitem{zhou2017unified}
	{\sc Z.~Zhou and A.~M.-C. So}, {\em A unified approach to error bounds for
		structured convex optimization problems}, Mathematical Programming, 165
	(2017), pp.~689--728.
	
	\bibitem{zou2006adaptive}
	{\sc H.~Zou}, {\em The adaptive lasso and its oracle properties}, J. American
	statistical association, 101 (2006), pp.~1418--1429.
	
\end{thebibliography}

\begin{thebibliography}{1}
	
	\bibitem{kummer1988newton}
	{\sc B.~Kummer}, {\em Newton's method for non-differentiable functions},
	Advances in Mathematical Optimization, 45 (1988), pp.~114--125.
	
	\bibitem{li2018efficiently}
	{\sc X.~Li, D.~F. Sun, and K.-C. Toh}, {\em On efficiently solving the
		subproblems of a level-set method for fused lasso problems}, SIAM J. on
	Optimization, 28 (2018), pp.~1842--1866.
	
	\bibitem{mifflin1977semismooth}
	{\sc R.~Mifflin}, {\em Semismooth and semiconvex functions in constrained
		optimization}, SIAM J. on Control and Optimization, 15 (1977), pp.~959--972.
	
	\bibitem{qi1993nonsmooth}
	{\sc L.~Qi and J.~Sun}, {\em A nonsmooth version of {Newton's} method},
	Mathematical Programming, 58 (1993), pp.~353--367.
	
	\bibitem{rockafellar2009variational}
	{\sc R.~T. Rockafellar and R.~J.-B. Wets}, {\em Variational {A}nalysis},
	vol.~317, Springer Science \& Business Media, 2009.
	
	\bibitem{sun2002semismooth}
	{\sc D.~F. Sun and J.~Sun}, {\em Semismooth matrix-valued functions},
	Mathematics of Operations Research, 27 (2002), pp.~150--169.
	
\end{thebibliography}
\end{document}